\documentclass[10pt,reqno]{amsart}
\usepackage{amsmath,amsfonts,amssymb,mathabx}
\usepackage{amsthm}
\usepackage{amstext}
\usepackage[normalem]{ulem}
\usepackage{verbatim}
\usepackage{longtable}
\usepackage{float}
\usepackage{makecell}
\usepackage{boldline}
\usepackage[title]{appendix}
\usepackage{caption}
\usepackage{array, makecell}
\usepackage{nicefrac}
\usepackage{siunitx}
\usepackage{indentfirst}
\usepackage{csquotes}

\newcolumntype{C}[1]{>{\centering\arraybackslash$}p{#1}<{$}}
\usepackage{multicol}
\usepackage{soul}

\usepackage[ruled,vlined]{algorithm2e}
\makeatletter
\renewcommand{\algocf@name}{}     

\makeatother

\setcellgapes{3pt}

\usepackage{fancyhdr} 
\usepackage{graphicx}  
\usepackage{bbm}
\usepackage{geometry}
\usepackage{enumitem}
\usepackage{hyperref}
\hypersetup{
	colorlinks,
	linkcolor={blue},
	citecolor={red},
	urlcolor={blue}
}
\usepackage{cancel}
\makeatletter
\newcommand{\setword}[2]{ 
	\phantomsection
	#1\def\@currentlabel{\unexpanded{#1}}\label{#2} 
}
\makeatother
\usepackage[usenames,dvipsnames]{color}

\newcommand{\adc}[1]{\textcolor{blue}{#1}}

\renewcommand{\labelenumi}{(\roman{enumi})}

\newtheorem{thm}{Theorem}[section]

\newtheorem{cor}[thm]{Corollary}
\newtheorem{rem}[thm]{Remark}
\newtheorem{rems}[thm]{Remarks}

\newtheorem{dfs}[thm]{Definitions}
\newtheorem{ex}[thm]{Example}

\newtheorem{symbs}[thm]{Notations}

\newtheorem{sthm}{Theorem}[subsection]
\newtheorem{srem}[sthm]{Remark}

\newtheorem{sex}[sthm]{Example}
\newcommand{\1}{\mathbbm{1}}
\newcommand{\wh}{\widehat}
\newcommand{\wt}{\widetilde}

\newcommand{\vS}{\varSigma}
\newcommand{\vO}{\varOmega}

\newcommand{\vY}{\varUpsilon}

\newcommand{\uph}{\upharpoonright}

\newcommand{\s}{\sigma}
\newcommand{\sq}{\subseteq}

\newcommand{\B}{\mathfrak B}
\newcommand{\N}{\mathbb N}

\newcommand{\R}{\mathbb R}

\newcommand{\E}{\mathbb E}

\newcommand{\F}{\mathcal F}
\newcommand{\C}{\mathcal C}
\newcommand{\G}{\mathcal G}
\newcommand{\D}{\mathcal D}
\newcommand{\M}{\mathcal M}

\def\H{{\mathcal H}}
\def\L{{\mathcal L}}

\newcommand{\al}{\alpha}
\newcommand{\ga}{\gamma}
\newcommand{\be}{\beta}
\newcommand{\ka}{\kappa}

\DeclareSymbolFont{extraup}{U}{zavm}{m}{n}
\DeclareMathSymbol{\varheart}{\mathalpha}{extraup}{86}
\DeclareMathSymbol{\vardiamond}{\mathalpha}{extraup}{87}

\title[A characterization of ruin-inducing probability measures in a renewal risk model]{A characterization of ruin-inducing probability measures in a renewal risk model}

\author[S.~M. Tzaninis]{Spyridon M. Tzaninis\textsuperscript{1}}
\address{\textsuperscript{1}Department of Statistics and Actuarial-Financial Mathematics\\ University of the Aegean\\ Karlovassi\\ GR-83 200 Samos\\ Greece}
\email{stzaninis@aegean.gr}
\thanks{}

\author[A. Bozikas]{Apostolos Bozikas\textsuperscript{2}}
\address{\textsuperscript{2}Department of Statistics and Insurance Science\\ University of Piraeus\\ 80 Karaoli and Dimitriou str.\\ 185 34 Piraeus\\ Greece} 
\email{bozikas@unipi.gr} 
\thanks{}

\date{\today}

\usepackage[most]{tcolorbox}

\newtcolorbox[auto counter]{myblock}[2][]{%
  breakable,
  enhanced,
  colback=white,
  colframe=black,
  fonttitle=\bfseries,
  coltitle=black,
  sharp corners,
  boxrule=0.5pt,
  attach title to upper,
  title={\underline{Algorithm. #2}}, 
  #1
}

\begin{document}
	
\begin{abstract}
In this work, we derive a complete characterization of all ruin-inducing probability measures that preserve the structure of a given compound renewal process in terms of suitable pairs of functions \((\ga,\delta)\). This result allows us to obtain an explicit representation of the infinite-time ruin probability as an expectation under any ruin-inducing probability measure. A key feature of our approach is that the construction of these measures does not rely on the existence of moment generating functions, and is therefore applicable to heavy-tailed claim size distributions. The proposed framework includes the classical  Esscher transform as a special case.
\end{abstract}
	
\maketitle
	
\section{Introduction and motivation}\label{intro}
Let $(\vO,\vS,P)$ be a probability space, and consider a (risk) reserve process $R^u{\,:=\,}\{R^u_t\}_{t\in\R_+}$ defined by 
\[
R^u_t:=u+c\cdot t-\sum_{n=1}^{N_t} X_n=:u+c\cdot t-S_t
\]	
for every $t{\,\in\,}\R_+$, where $u{\,\in\,}\R_+$ is the \textit{initial capital}, $c{\,\in\,}(0,\infty)$ is the \textit{premium rate}, $N{\,:=\,}\{N_t\}_{t\in\R_+}$ is a counting process, $X{\,:=\,}\{X_n\}_{n\in\N}$ is  the  claim size process, and $S{\,:=\,}\{S_t\}_{t\in\R_+}$ is the aggregate claims process induced by $N$ and $X$. The function $\psi{\,:\,}\R_+{\,\to\,} [0,1]$ defined by $\psi(u){\,:=\,}P\big(\{\inf_{t\in\R_+}R^u_t{\,<\,}0\}\big)$ for any $u{\,\in\R_+\,}$ is the  (infinite-time) ruin probability for $R^u$ with respect to $P$. Also, let $\tau_u{\,:=\,}\inf\{t{\,\in\,}\R_+ : R^u_t{\,<\,}0\}$ ($\inf\varnothing{\,=\,}\infty$) denote the ruin time.  
\smallskip
  
In the  Cram\'{e}r-Lundberg and Sparre Andersen risk models, explicit formulas for $\psi(u)$ are available only in special cases (e.g. if the claim sizes follow a phase-type distribution, see Asmussen \& Albrecher~\cite{asal}, Chapter~IX, for an overview). This motivates the use of simulation-based approaches. For the finite-time ruin probability
\[
\psi(u,t):=P\big(\{\inf_{0\leq s\leq t } R^u_s<0\}\big)=P\big(\{\tau_u\leq t\}\big) 
\]
for all $t,u{\,\in\,}\R_+$,  the indicator function $\1_{\{\tau_u\leq t\}}$ can be   simulated via \textit{crude Monte Carlo} (MC).  However, this method is not directly applicable to \(\psi(u)\), since the event  $\{\tau_u {\,<\,} \infty\}$ cannot be simulated within finite time (cf., e.g., Asmussen \& Albrecher~\cite{asal}, p.~462). This limitation reflects a deeper structural issue: one seeks a new probability measure  under which ruin occurs almost surely (a.s.), while also preserving the probabilistic structure of the model. Once such a probability measure is constructed, simulation techniques such as \textit{Importance Sampling} (IS) can be employed (see Juneja \& Shahabuddin~\cite{jus} for an overview) to estimate \(\psi(u)\). \smallskip 

The construction of such measures is usually done by the \textit{exponential change of measures} technique. Let $\mathcal G{\,:=\,}\{\mathcal G_t\}_{t\in\R_+}$ be a filtration for $\vS$ and let $Y{\,:=\,}\{Y_t\}_{t\in\R_+}$ be a stochastic process on $(\vO,\vS)$ adapted to $\G$.  If there exists a $(P,\mathcal G)$-martingale $Z{\,:=\,}\{Z_t\}_{t\in\R_+}$ with \(Z_t{\,>\,}0\) \(P\)-a.s. and $\E_P[Z_t]{\,=\,}1$ for all $t{\,\in\,}\R_+$, then this can be used  as a likelihood ratio  to define a new probability measure $Q$ on $\vS$  by 
\[
Q(A):=\E_P[\1_A\cdot Z_t]\,\,\text{ for every } A \in\mathcal G_t.
\] 
In problems involving change of measures, the process \(Y\) is usually Markovian under \(P\). In such cases one can construct a new probability measure  \(Q\) so that the process still has the Markov property under $Q$  (see Palmowski \& Rolski~\cite{paro}). However,  when $Y$ is not Markov under $P$, the situation becomes more delicate. Typically, one must first markovize  the process through a supplementary variable, an idea tracing back to Cox~\cite{cox}, and then apply either the theory of Piecewise Deterministic Markov Processes (PDMPs), see Davis~\cite{davis}, or  of Markov Additive Processes (MAPs), see Asmussen \& Kella~\cite{akel}, to construct  suitable exponential martingales.\smallskip 
	
In the Sparre Andersen risk model,  where $S$ is a $P$-compound renewal process (CRP), the reserve process $R^u$ fails in general to be Markov (except in the compound Poisson case). Thus, one must first markovize \(R^u\), either by introducing the backward recurrence time $J_t$ or the forward recurrence time $V_t$ (cf., e.g., Rolski et al.~\cite{rss}, Sections~11.5.2 and~11.5.3, respectively). Following this, one may construct an a.s. positive $(P,\F)$-martingale (see  e.g., Dassios \& Embrechts~\cite{daem}, Theorem~10) or an a.s. positive $(P,\H)$-martingale (see e.g., Dassios~\cite{das}, Theorem~1.1.3), where $\F$ is the natural filtration of $S$, and  $\H$ is the filtration generated by $S$ and $V{\,:=\,}\{V_t\}_{t\in\R_+}$. Note here that  there exist cases where $\H$ is preferred to the natural filtration  $\F$ e.g., for a delayed renewal process  (cf., e.g., Schmidli~\cite{scm}, pp.~178-179). Using either of these  martingales, one can  define a new probability measure $Q$ on $\vS$ so that  $S$ remains a $Q$-CRP and  ruin occurs $Q$-a.s.. This approach, which is closely related to exponential tilting and the Esscher transform, has been widely used  in the literature to tackle various ruin-related problems (see, for example, Constantinescu et al.~\cite{cdnp},  Dassios~\cite{das}, Dassios \& Embrechts \cite{daem},  Schmidli~\cite{scm95, scm10} and Tzaninis~\cite{t1}),  as it resolves the infinite time horizon problem, and allows us to express $\psi(u)$ as a quantity under the new measure via the equality
\[
\psi(u)=\E_P[\1_{\{\tau_u<\infty\}}]=\E_Q[1/M_{\tau_u}]=\E_Q[1/L_{\tau_u}]  \,\,\text{ for every } u\in\R_+.
\]   
However, this construction relies fundamentally on the existence of moment generating functions (mgf) and therefore excludes heavy-tailed claim size distributions.  \smallskip
   
In the heavy-tailed setting, the main approaches  for computing \(\psi(u)\) are based on   
the widely used \textit{Pollaczek-Khinchin} (PK) formula (cf., e.g., Rolski et al.~\cite{rss}, Sections 5.3.3 and 6.5.1, for more details).  In this context, simulation methods such as  IS  and conditional MC are typically employed (see, for example, Asmussen \& Kroese~\cite{askr06}, Asmussen et al.~\cite{askrru05}, Juneja \& Shahabuddin~\cite{jusha02}, and references therein).~Nevertheless, simulation methods based on the PK formula  are essentially restricted to the Cram\'er-Lundberg model, where the ladder height distribution admits an explicit form. 
For an approach based on the PK formula in a Sparre Andersen risk model, with completely monotone claim size distributions and interarrival time distributions having rational Laplace transform, we refer to Albrecher \& Vatamidou~\cite{alva19}.\smallskip 
	
The primary objective of this paper is  to overcome these limitations. To this end, we provide a complete characterization of all ruin-inducing probability measures $Q$ on $\vS$ that are progressively equivalent to $P$ on \(\H\) and preserve the structure of a given \(P\)-CRP. In particular, we obtain a martingale representation for $\psi(u)$ 
that allows us to move beyond the classical adjustment coefficient approach and to provide a rigorous and unified framework for simulating ruin probabilities in a Sparre Andersen risk model.\smallskip 

In Section~\ref{CRPPEM}, we first provide a characterization of all probability measures $Q$  that are progressively equivalent to \(P\) on $\H$ and preserve the structure of a given $P$-CRP (see Theorem~\ref{thm!}). This result offers an elementary method for constructing general exponential martingales without relying on the complex machinery of PDMP or MAP theory (see also  Corollary~\ref{fmar}).  The methodological contribution of Theorem~\ref{thm!} is conceptually distinct from the classical PDMP or MAP approaches. In generator-based methods, martingales are constructed via the infinitesimal generator, and the resulting claim size and interarrival time distributions under the new measure are determined implicitly through the associated operator equations. By contrast, Theorem~\ref{thm!} enables a ``reverse-engineering'' approach: a researcher or practitioner may prescribe the desired target distributions for claim sizes and interarrival times under the new measure $Q$, and the theorem identifies an essentially unique tilting pair $(\gamma,\delta)$ together with the corresponding martingale, which arises naturally as a Radon-Nikod\'{y}m derivative. This transparency not only simplifies the construction but also ensures that the resulting measure is structurally well-defined and immediately applicable. In Section~\ref{app}, we prove the desired characterization of ruin-inducing probability measures, see Theorem~\ref{gform}. This allows us to obtain a martingale representation for \(\psi(u)\) suitable for simulation purposes (see condition~\eqref{ruin} in Theorem~\ref{gform}). In Section~\ref{examples}, we present various theoretical and numerical examples including both light- and heavy-tailed distributions for the claim sizes and the interarrival times. Finally, Section~\ref{cr} concludes the paper.  
  
	
\section{A change of measures technique for compound renewal processes}\label{CRPPEM}
	
$\N$ and $\R$ stand for the natural and the real numbers, respectively, while $\N_0{\,:=\,}\N{\,\cup\,}\{0\}$ and $\R_+{\,:=\,}\{x \in\R : x\geq0\}$. For a map $f{\,:\,}D{\,\to\,} E$ and for a non-empty set $A{\,\subseteq\,}{D}$  write $f{\uph}{A}$ for the restriction of $f$ to $A$ and $\1_A$ for the indicator function of the set $A$. For every $D{\,\subseteq\,} \R$ denote by $\B(D)$ the $\sigma$-algebra on $D$ generated by the topology $\mathfrak{E}{\,\cap\,} D{\,:=\,}\{D{\,\cap\,} E {\,:\,} E{\,\in\,}\mathfrak E\}$, where $\mathfrak E$ is the Euclidean topology over $\R$. In particular, for all $-\infty{\,\leq\,}\al{\,<\,}\be{\,\leq\,}+\infty$, write $\B(\al,\be){\,:=\,}\B\big((\al,\be)\big)$ for simplicity.   
\smallskip
	
Let $(\vO,\vS,P)$ be a probability space. For the definitions of a \textit{counting} (or \textit{claim number}) \textit{process} $N:=\{N_t\}_{t\in\R_+}$, a \textit{(claim) interarrival process} $W{\,:=\,}\{W_n\}_{n\in\N}$ and a \textit{(claim) arrival process}  $T{\,:=\,}\{T_n\}_{n\in\N_0}$ see, e.g., Schmidt~\cite{Sc}, pp.~17 and~6-7, respectively.  Recall that if   $W$ is $P$-i.i.d., then the counting process $N$ is a  \textit{renewal process} under $P$  (written $P$-RP  for short). Note that if $N$ is a $P$-RP then $\mathbb E_P[N^r_t]{\,<\,}\infty$ for any $t{\,\in\,}\R_+$ and $r{\,>\,}0$ (cf., e.g., Gut~\cite{gut}, Chapter 2, Theorem~3.1(ii)); hence according to Schmidt~\cite{Sc}, Corollary 2.1.5, $N$ has zero probability of explosion, that is $P\big(\{\sup_{n\in\N} T_n{\,<\,}\infty\}\big){\,=\,}0$.  Furthermore, if $X{\,:=\,}\{X_n\}_{n\in\N}$ is another sequence of $P$-i.i.d positive real-valued random variables on $\vO$, called a \textit{claim size process} (cf., e.g., Schmidt~\cite{Sc}, p.~103), which is $P$-mutually independent of $N$, define the \textit{aggregate claims process} $S{\,:=\,}\{S_t\}_{t\in\R_+}$ by means of $S_t{\,:=\,}\sum_{n=1}^{N_t} X_n$ for any $t{\,\in\,}\R_+$ (cf., e.g., Schmidt~\cite{Sc},  p.~103). In particular, if $N$ is a $P$-RP, then the aggregate claims process is a  \textit{compound renewal process} under $P$ (written $P$-CRP for short). \smallskip
	
	\textit{In what follows, put $\F^W{\,:=\,}\{\F^W_n\}_{n\in\N}$, $\F^X{\,:=\,}\{\F^X_n\}_{n\in\N}$ and $\F{\,:=\,}\{\F_t\}_{t\in\R_+}$ for the natural filtrations of $W, X$ and $S$, respectively.}\smallskip
	
For the definition of a $(P, \mathcal{Z})$-martingale, where $\mathcal{Z}{\,:=\,}\{\mathcal{Z}_t\}_{t\in\R_+}$ is a filtration for $\vS$, we refer to Schmidt~\cite{Sc}, p.~25. A $(P, \mathcal{Z})$-martingale $\{Z_t\}_{t\in\R_+}$ is $P$-\textit{a.s. positive} if, $Z_t{\,>\,}0$ is $P$-a.s. for each $t\in\R_+$.  \smallskip
	
	For any $t{\,\in\,}\R_+$  denote by $J_t$ the \textit{backward recurrence time} and by $V_t$  the \textit{forward recurrence time}, i.e. $J_t{\,:=\,}t-T_{N_t}$ and $V_t{\,:=\,}T_{N_t+1} - t$, respectively. Furthermore, denote by $\F^V{\,:=\,}\{\F^V_t\}_{t\in\R_+}$ the natural filtration generated by $V{\,:=\,}\{V_t\}_{t\in\R_+}$, where $\F^V_t{\,:=\,}\sigma\Big(\bigcup_{u\leq t}\sigma(V_u)\Big)$ for any $t{\,\in\,}\R_+$, and put $\H{\,:=\,}\{\H_t\}_{t\in\R_+}$ for the progressive enlargement of the filtration $\F$ by $\F^V$, that is  $\H_t{\,:=\,}\sigma(\F_t\cup\F^V_t)$ for every $t{\,\in\,}\R_+$.  Clearly \(\H_\infty{\,:=\,} \sigma(\bigcup_{t\in\R_+}\H_t){\,=\,} \sigma(\bigcup_{t\in\R_+}\F_t) {\,=:\,} \F_\infty\). 
	
	\begin{dfs}\label{df0}
		\normalfont
		Let $Q$ be a probability measure on $(\vO,\vS)$.   \smallskip
		
		\noindent\textbf{(a)} If $\G{\,\subseteq\,}\vS$ is a $\s$-algebra of subsets of $\vO$, then  $Q$ is \textit{absolutely continuous}  with respect to $P$ on $\G$ (in symbols $Q{\uph}\G{\,\ll\,}P{\uph}\G$) if
		\[
		\forall\; A\in\G\quad  \Big( P(A)=0\Rightarrow Q(A)=0 \Big).
		\] 
		In addition, if $P{\uph}\G{\,\ll\,} Q{\uph}\G$, then $Q$ and $P$ are \textit{equivalent} on $\G$ (in symbols $Q{\uph}\G{\,\sim\,}P{\uph}\G$). For $\G{\,=\,}\vS$ simply write $Q{\,\ll\,} P$ and $Q{\sim}P$.  
		\smallskip
		
		\noindent  	\textbf{(b)} If $\{\G_t\}_{t\in\R_+}$ is a filtration for $\vS$, then $Q$ and $P$ are \textit{progressively equivalent} on $\{\G_t\}_{t\in\R_+}$, if $Q$ and $P$   are equivalent on each  $\G_t$ (i.e., $Q{\uph}\G_t{\,\sim\,} P{\uph}\G_t$). If $P$ and $Q$  are progressively equivalent with respect to $\{\G_t\}_{t\in\R_+}$, we write
		\[
		\frac{dQ}{dP}{\uph}\G_t:=\frac{d(Q{\uph}\G_t)}{d(P{\uph}\G_t)} \,\,\text{ for all }  t\in\R_+
		\]
		for simplicity.
	\end{dfs}
	
 For reasons of mathematical convenience, we  introduce a new basic probability space  $(\vO,\vS,P)$. To this purpose, put $(\vO,\vS){\,:=\,} \big((0,\infty)^\N{\,\times\,}(0,\infty)^\N,\B(\vO)\big)$ and let  $\mu$ and $\nu$ be probability measures on $\B(0,\infty)$. Define the probability measure $P{\,:=\,}\mu_\N{\,\otimes\,}\nu_\N$ on $\vS$, and put $W_n(\omega){\,:=\,}w_n$, $X_n(\omega){\,:=\,}x_n$ for every $\omega{\,=\,}(w_1,\ldots,w_n,\ldots ; x_1,\ldots, x_n,\ldots){\,\in\,}\vO$. We then get that $W{\,:=\,}\{W_n\}_{n\in\N}$ and $X{\,:=\,}\{X_n\}_{n\in\N}$ are two $P$-mutually independent sequences of $P$-i.i.d. positive real-valued random variables,   satisfying conditions $P_{W_1}{\,=\,}\mu$ and $P_{X_1}{\,=\,}\nu$.  Putting $T_n{\,:=\,}\sum_{k=1}^n W_k$ for any $n{\,\in\,}\N_0$ and $T{\,:=\,}\{T_n\}_{n\in\N_0}$,  we define the counting process $N{\,:=\,}\{N_t\}_{t\in\R_+}$ by means of $N_t{\,:=\,}\sum_{n=1}^\infty\1_{\{T_n\leq t\}}$ for any $t{\,\in\,}\R_+$  (cf., e.g., Schmidt~\cite{Sc}, Theorem~2.1.1). Setting now $S_t{\,:=\,}\sum^{N_t}_{n=1} X_n$ for any $t{\,\in\,}\R_+$ and $S{\,:=\,}\{S_t\}_{t\in\R_+}$, it follows that \(S\) is a \(P\)-CRP. Clearly,  $\vS{\,=\,}\F_\infty^{(W,X)}{\,=\,}\F_\infty{\,=\,}\H_\infty$. 
\smallskip	
	
In order to present the main result of this section  we have to introduce the following classes.

\begin{symbs}\label{symb2} 
\normalfont 
{\bf(a)}  The class of all real-valued $\B(0,\infty)$-measurable functions $\ga$ with \(\E_{P}\big[e^{\ga(X_{1})}\big]{\,=\,}1\) will be denoted by $\hypertarget{fp}{\F_{P}}{\,:=\,}\F_{P,X}$.  Furthermore, denote by  $\hypertarget{gp}{\G_{P}}{\,:=\,}\G_{P,W}$  the class consisting of all  real-valued $\B(0,\infty)$-measurable functions $\delta$  satisfying condition $\E_P\big[e^{\delta(W_1)}\big]{\,=\,}1$.\smallskip    
		
\noindent 	{\bf (b)}   The class of all probability measures $Q$ on $\vS$ so that \(S\) is  a \(Q\)-CRP and   $Q{\uph}\H_t{\,\sim\,} P{\uph}\H_t$ for all $t{\,\in\,}\R_+$, will be denoted by   $\hypertarget{ms}{\M^\H_S}{\,:=\,}\M^\H_{S,P,W,X}$.
	\end{symbs}	
	 
The next result, which improves Macheras \& Tzaninis~\cite{mt3}, Theorem~3.1, plays a key role in the proof of our main result, as it establishes a bijection  between the elements of  $\hyperlink{ms}{\M^\H_S}$ and  $\hyperlink{fp}{\F_P}{\,\times\,}\hyperlink{gp}{\G_P}$. For its formulation,  for every  pair $(\ga,\delta){\,\in\,}\hyperlink{fp}{\F_P}{\,\times\,}\hyperlink{gp}{\G_P}$  define  $S_t^{(\ga)}{\,:=\,}\sum_{j=1}^{N_t}\ga(X_j)$ and $T_{N_t}^{(\delta)}{\,:=\,}\sum_{j=1}^{N_t}\delta(W_j)$  for all $t{\,\in\,}\R_+$.

\begin{thm}\label{thm!}  
The following assertions hold:
\begin{enumerate}
\item
for every  $Q{\,\in\,}\hyperlink{ms}{\M^\H_S}$ there exists an essentially unique pair $(\ga,\delta){\,\in\,}\hyperlink{fp}{\F_P}{\,\times\,}\hyperlink{gp}{\G_P}$ so that the conditions 
\begin{gather}
\tag{$RM(\H)$}
\frac{dQ}{dP}{\uph}\H_t=e^{S_t^{(\ga)}+T_{N_t+1}^{(\delta)}}=:L^{(\ga,\delta)}_{t}\quad P{\uph}\H_t\text{-a.s.} 
\label{mart}
\end{gather}
and 
\begin{gather}
\tag{$RM(\F)$}
\frac{dQ}{dP}{\uph}\F_t=e^{S_t^{(\ga)}+T_{N_t}^{(\delta)}}\cdot\frac{\int_{J_t}^\infty e^{\delta(w)}\,P_{W_1}(dw)}{\int_{J_t}^\infty \,P_{W_1}(dw)}=:M^{(\ga,\delta)}_{t}\quad P{\uph}\F_t\text{-a.s.}
\label{mart1}
\end{gather}
hold for all $t{\,\in\,}\R_+$, and  the processes $L^{(\ga,\delta)}{\,:=\,}\{L^{(\ga,\delta)}_{t}\}_{t\in\R_+}$ and $M^{(\ga,\delta)}{\,:=\,}\{M^{(\ga,\delta)}_{t}\}_{t\in\R_+}$ are  a.s. positive $(P,\H)$- and $(P,\F)$-martingales, respectively, satisfying $\E_P\big[L^{(\ga,\delta)}_{t}\big]{\,=\,}\E_P\big[M^{(\ga,\delta)}_{t}\big]{\,=\,}1$; 
\item
conversely, for every pair $(\ga,\delta){\,\in\,}\hyperlink{fp}{\F_P}{\,\times\,}\hyperlink{gp}{\G_P}$ there exists a unique probability measure $Q{\,:=\,}Q^{(\ga,\delta)}{\,\in\,}\hyperlink{ms}{\M^\H_S}$ determined by 
\begin{gather}		
Q(A)=\E_P\big[\1_A\cdot L^{(\ga,\delta)}_{t}\big]\,\,\text{ for every } t \in \R_+\text{ and } A \in \H_t
\label{mq}
\tag{$\ast$}
\end{gather}
and satisfying condition \eqref{mart1}.
\end{enumerate}
\end{thm}
	
\begin{proof}
Fix on arbitrary $t{\,\in\,}\R_+$ and $n{\,\in\,}\N_0$.\medskip 
		
\noindent Ad (i): Let $Q{\,\in\,}\hyperlink{ms}{\M^\H_S}$. First note that since $\mathcal{F}_t{\,\subseteq\,}\H_t$, we easily get $Q{\uph}\F_t{\,\sim\,}P{\uph}\F_t$, which, together with Macheras \& Tzaninis~\cite{mt3}, Proposition~2.1(i)${\,\Rightarrow\,}$(ii), implies that $Q_{W_1}{\,\sim\,}P_{W_1}$ and $Q_{X_1}{\,\sim\,}P_{X_1}$. The rest of the proof is divided into the following six steps. \smallskip
		
\noindent {\bf (a)} There exists a $P_{W_1}$-a.s. unique  function $\delta{\,\in\,}\hyperlink{gp}{\G_P}$, defined by means of $\delta{\,:=\,}\ln g$, where $g$ is a  $P_{W_1}$-a.s. positive Radon-Nikod\'{y}m derivative  of $Q_{W_1}$ with respect to $P_{W_1}$, satisfying the condition 
\begin{gather}
Q(E)=\E_{P}\big[\1_E\cdot e^{\sum_{j=1}^{n+1}\delta(W_j)}\big]\quad\text{ for any }  E\in\F^W_{n+1}.
\label{rnw}
\end{gather}
		
In fact, since $P_{W_1}{\,\sim\,} Q_{W_1}$,  it follows by the Radon-Nikod\'{y}m Theorem that there exists a $P_{W_1}$-a.s. unique positive Radon-Nikod\'{y}m derivative $g$ of $Q_{W_1}$ with respect to $P_{W_1}$. Put $\delta{\,:=\,}\ln g$. It then follows that $\delta{\,\in\,}\hyperlink{gp}{\G_P}$.  To check the validity of condition \eqref{rnw}, consider the family of sets  ${\C}^W_{n+1}{\,:=\,}\left\{\bigcap^{n+1}_{j=1} C_j {\,:\,} C_{j}{\,\in\,}\sigma(W_{j})\right\}$ which is  a generator of $\mathcal{F}_{n+1}^W$, closed under finite intersections. An easy computation shows that any $C{\,\in\,}\C^W_{n+1}$ satisfies condition \eqref{rnw}. The latter, together with Dynkin's Lemma, proves (a). \smallskip
		
\noindent {\bf (b)} There exists a $P_{X_1}$-a.s. unique  function $\ga{\,\in\,}\hyperlink{fp}{\F_P}$, defined by means of $\ga{\,:=\,}\ln f$, where $f$ is a  $P_{X_1}$-a.s. positive Radon-Nikod\'{y}m derivative  of $Q_{X_1}$ with respect to $P_{X_1}$, satisfying the condition 
\begin{gather}
Q(F)=\E_{P}\big[\1_F\cdot e^{\sum_{j=1}^n \ga(X_j)}\big]\quad\text{ for any }   F\in\F^X_n.
\label{d1}
\end{gather}
		
In fact, since $P_{X_1}{\,\sim\,}{Q}_{X_1}$, then similarly to (a)  there exists a $P_{X_1}$-a.s. unique positive Radon-Nikod\'{y}m derivative $f$ of $Q_{X_1}$ with respect to $P_{X_1}$. Putting now $\ga{\,:=\,}\ln f$, we immediately get $\ga{\,\in\,}\hyperlink{fp}{\F_P}$. The remainder of the proof of (b) follows by arguments similar to those used in (a).  \smallskip
		
\noindent {\bf (c)} Condition    
\begin{gather}
Q(G)=\E_P\big[\1_G\cdot e^{\sum_{j=1}^n \ga(X_j)+\sum_{j=1}^{n+1} \delta(W_j)}\big] \quad \text{ for any } G\in\G_{n,n+1}
\label{d2}
\end{gather}
holds true, where $\G_{n,n+1}{\,:=\,}\sigma(\F^X_n\cup\F^W_{n+1})$.\smallskip
		
		In fact, consider the family  
		\[
		\C_{n,n+1}:=\Bigg\{\bigg(\bigcap_{j=1}^{n}  X_j^{-1}[F_j]\bigg)\cap\bigg(\bigcap_{j=1}^{n+1} W_{j}^{-1}[E_{j}]\bigg)   : F_j, E_j\in\B(0,\infty)\Bigg\}
		\]
		and let $G{\,\in\,}\C_{n,n+1}$. Since $X$ and $W$ are $Q$-mutually independent, we may apply  (a) and (b) to get 
		\[
		Q(G)=\E_P\big[\1_G\cdot e^{\sum_{j=1}^n \ga(X_j)+\sum_{j=1}^{n+1} \delta(W_j)}\big].
		\]
		Denote by  $\D_{n,n+1}$ the family of all $G{\,\in\,}\G_{n,n+1}$  satisfying condition \eqref{d2}. Clearly, $\D_{n,n+1}$ is a Dynkin class.  Since $\C_{n,n+1}$ is a   generator of $\G_{n,n+1}$ which is closed under finite intersections and  $\C_{n,n+1}{\,\subseteq\,} \D_{n,n+1}$, we can apply Dynkin's Lemma  to  complete the proof of (c).\smallskip
		
		\noindent \textbf{(d)}  The equality $\H_t{\,\cap\,}\{N_t{\,=\,}n\}{\,=\,}\G_{n,n+1}{\,\cap\,}\{N_t=n\}$ holds true. \smallskip  
		
		In fact, since the inclusion $\G_{n,n+1}{\,\cap\,}\{N_t=n\}{\,\sq\,}\H_t{\,\cap\,}\{N_t{\,=\,}n\}$  is immediate,  it suffices to prove only the inverse inclusion.  To this end,   let $A{\,\in\,}\bigcup_{0\leq{u}\leq{t}}\s(V_u)$ be arbitrary. There exists an index $u{\,\in\,}[0,t]$ and a set $B{\,\in\,}\B(0,\infty)$ such that  
		\[
		A\cap\{N_t=n\} =\bigg[\bigcup_{k=0}^\infty\Big(\{N_u=k\}\cap  \big(T_{k+1}-u\big)^{-1}[B]\Big)\bigg]\cap\{N_t=n\} = D_{n,u}\cap\{N_t=n\},
		\]
		where 
		\[
		D_{n,u}:=\bigcup_{k=0}^n\Big(\{N_u=k\}\cap \big(T_{k+1}-u\big)^{-1}[B]\Big)\in\F^W_{n+1}\sq\G_{n,n+1};
		\]
		hence	$\Big(\bigcup_{0\leq u\leq{t}}\s(V_u)\Big){\,\cap\,}\{N_t{\,=\,}n\}{\,\sq\,}\G_{n,n+1}{\,\cap\,}\{N_t=n\}$. This, together with condition 
		\[
		\F_t\cap\{N_t=n\}=\G_{n,n}\cap\{N_t=n\}\subseteq\G_{n,n+1}\cap\{N_t=n\},
		\]
		where the equality follows by e.g. Jacod~\cite{jacod}, Proposition~3.39(a),   completes the proof of (d).\smallskip
		
		\noindent {\bf (e)} Condition \eqref{mart} holds true and the family $L^{(\ga,\delta)}$ is an a.s. positive $(P,\H)$-martingale satisfying   $\E_P\big[L_t^{(\ga,\delta)}\big]{\,=\,}1$.  \smallskip
		
		In fact, let  $A{\,\in\,}\H_t$ be given.  Since by step (d) there exists a set $B_{n,n+1}{\,\in\,}\G_{n,n+1}$ such that 
		\[
		A\cap\{N_t=n\}=B_{n,n+1}\cap\{N_t=n\},
		\]
		it follows that
		\[
		Q(A)=\sum_{k\in\N_0} Q(A\cap\{N_t=k\})
		=\sum_{k\in\N_0} Q(B_{k,k+1}\cap  \{ N_t =k\}),
		\]
		where the first equality follows since $N$ has zero probability of explosion. But since  $B_{n,n+1}{\,\cap\,}\{N_t{\,=\,}n\}{\,\in\,}\G_{n,n+1}$, one gets by step (c) that 
		\[
		Q(A)=\sum_{k\in\N_0}\E_{P}\big[\1_{B_{k,k+1}\cap  \{ N_t =k\}}\cdot e^{\sum_{j=1}^k \ga(X_j)+\sum_{j=1}^{k+1} \delta(W_j)}\big] =\E_{P}\big[\1_A\cdot e^{S_t^{(\ga)}+T^{(\delta)}_{N_t+1}}\big];
		\]
		hence
		\begin{gather}
			Q(A)=\E_{P}\big[\1_A\cdot L^{(\ga,\delta)}_{t}\big]\,\text{ for all } A\in\H_{t},
			\label{21}
		\end{gather}
		where $L^{(\ga,\delta)}_{t}{\,:=\,}e^{S_t^{(\ga)}+T^{(\delta)}_{N_t+1}}$. Since condition \eqref{21} holds true for any $s{\,\in\,}[0,t]$ and $A{\,\in\,}\H_s$,  it follows that  $L^{(\ga,\delta)}{\,:=\,}\{L^{(\ga,\delta)}_s\}_{s\in\R_+}$  is a $(P,\H)$-martingale.  The latter along with condition~\eqref{21} proves~\eqref{mart}. Condition~\eqref{mart} for $A{\,=\,}\vO$ yields $\E_{P}\big[L^{(\ga,\delta)}_t\big]{\,=\,}Q(\vO){\,=\,}1$. Finally, since $e^\delta$ and $e^\ga$ are $P_{W_1}$- and $P_{X_1}$-a.s. positive functions, respectively, it follows that $P\big(\big\{L^{(\ga,\delta)}_t{\,>\,}0\big\}\big){\,=\,}1$.\smallskip
		
		\noindent {\bf (f)} Condition~\eqref{mart1} holds true and the family $M^{(\ga,\delta)}$ is an a.s. positive $(P,\F)$-martingale satisfying condition $\E_P\big[M_t^{(\ga,\delta)}\big]=1$.  \smallskip
		
		The proof follows from e.g. Jacod~\cite{jacod}, Proposition~3.39(a), together with arguments similar to those used in (e). This completes the proof of  (f), and hence of statement (i).\medskip 
		
\noindent Ad (ii): Let \(Q\) be a set-function defined by condition \eqref{mq}. For a given pair $(\ga,\delta){\,\in\,}\hyperlink{fp}{\F_P}{\,\times\,}\hyperlink{gp}{\G_P}$,   consider the set-functions  $\wt \mu, \wt \nu {\,:\,}\B(0,\infty){\,\to\,} \R$  defined by 
\[
\wt\mu(B){\,:=\,}\E_{P}\big[\1_{W^{-1}_{1}[B]}\cdot e^{\delta(W_1)}\,\big]\,\,\text{ and }\,\, \wt \nu(E){\,:=\,}\E_{P}\big[\1_{X^{-1}_{1}[E]}\cdot e^{\ga(X_1)}\big]
\]
for all $B, E{\,\in\,}\B(0,\infty)$. Clearly $\wt \mu$ and $\wt\nu$ are probability measures on $\B(0,\infty)$. Thus, we may define a probability measure $\wt Q{\,:=\,}\wt\mu_\N{\,\otimes\,}\wt\nu_\N$ on $\vS$ such that  $S$ is a $\wt{Q}$-CRP  with $\wt Q_{W_1}{\,=\,}\wt\mu$ and $\wt Q_{X_1}{\,=\,}\wt\nu$. Since $\wt Q_{W_1}{\,\sim\,}P_{W_1}$ and $\wt Q_{X_1}{\,\sim\,} P_{X_1}$, the arguments used in the first part of the proof of step~(e) yield
\[
\wt Q(A)=\E_{P}\big[\1_A\cdot L^{(\ga,\delta)}_{t}\big]\,\text{ for every } A\in\H_{t}.
\]
Consequently, $\wt Q{\uph}\H_t{\,=\,}Q{\uph}\H_t$, and therefore $\wt Q{\uph}\wt\vS{\,=\,}Q{\uph}\wt\vS$, where $\wt\vS{\,:=\,}\bigcup_{s\in\R_+}\H_s$. Thus,  $Q$ is $\sigma$-additive on $\wt\vS$, and  $\wt Q$ is the unique extension of $Q$ on $\vS{\,=\,}\sigma(\wt\vS)$.  In particular,  $Q{\,\in\,}\hyperlink{ms}{\M^\H_S}$, and  statement~(ii) follows.  
	\end{proof}

\begin{rems}
\label{remp}
\normalfont
\textbf{(a)} Let $(\ga,\delta){\,\in\,}\hyperlink{fp}{\F_P}{\,\times\,}\hyperlink{gp}{\G_P}$. An immediate consequence of condition \eqref{mq}  is that the  distributions of the claim sizes and claim interarrival times under $Q{\,\in\,}\hyperlink{ms}{\M^\H_S}$ are given by  
\[
Q_{X_1}(B)=\E_P\big[\1_{X_1^{-1}[B]}\cdot e^{\ga(X_1)}\big]\,\,\text { and }\,\, Q_{W_1}(E)=\E_P\big[\1_{W_1^{-1}[E]}\cdot e^{\delta(W_1)}\big]
\]
for any $B, E{\,\in\,}\B(0,\infty)$. Moreover, these conditions imply that $e^\ga$ and  $e^\delta$ are the \(P_{X_1}\)- and \(P_{W_1}\)-a.s. unique positive Radon-Nikod\'{y}m derivatives of $Q_{X_1}$ with respect to $P_{X_1}$ and of  $Q_{W_1}$ with respect to $P_{W_1}$, respectively.\smallskip 

\noindent \textbf{(b)}  Let $(\ga,\delta){\,\in\,}\hyperlink{fp}{\F_P}{\,\times\,}\hyperlink{gp}{\G_P}$. If $\ga{\,\neq\,} 0$ or $\delta{\,\neq\,}0$, then, according to Macheras \& Tzaninis~\cite{mt3}, Remark~3.1, the probability measure $Q$ constructed in Theorem~\ref{thm!}(ii) is singular to  $P$.  \smallskip 
 
\noindent \textbf{(c)} Let us  denote by \(\hypertarget{msF}{\M_S}\) the class of all probability measures \(Q\) on \(\vS\) that are progressively equivalent to \(P\) on \(\F\) and such that \(S\) is a \(Q\)-CRP. This is precisely the class of probability measures studied in Macheras \& Tzaninis~\cite{mt3}.  We now show that the classes $\hyperlink{ms}{\M^\H_S}$ and $\hyperlink{msF}{\M_S}$ coincide.\smallskip

Fix an arbitrary \(t{\,\in\,}\R_+\). Since \(\F_t{\,\subseteq\,}\H_t\), it follows that for every \(Q{\,\in\,}\hyperlink{ms}{\M^\H_S}\) 
\[
Q{\uph}\H_t\sim P{\uph}\H_t \Rightarrow Q{\uph}\F_t\sim P{\uph}\F_t,
\]
which implies that  $\hyperlink{ms}{\M^\H_S}{\,\subseteq\,} \hyperlink{msF}{\M_S}$.\smallskip 

Conversely,  let $Q{\,\in\,}\hyperlink{msF}{\M_S}$. By~\cite{mt3}, Proposition~2.1(i)$\Rightarrow$(ii), we have $Q_{X_1}{\,\sim\,} P_{X_1}$ and $Q_{W_1}{\,\sim\,}P_{W_1}$.  By the arguments used in the proof of Theorem~\ref{thm!}, we obtain
\[
\frac{dQ}{dP}{\uph}\H_t=\Big[\prod_{j=1}^{N_t}\frac{dQ_{X_1}}{dP_{X_1}}(X_j)\Big]\cdot \Big[\prod_{j=1}^{N_t+1} \frac{dQ_{W_1}}{dP_{W_1}}(W_j)\Big]\quad P{\uph}\H_t\text{-a.s..}
\]
and 
\[
P\Big(\Big\{\frac{dQ}{dP}{\uph}\H_t>0\Big\}\Big)=1.
\]
Therefore $Q{\uph}\H_t{\,\sim\,} P{\uph}\H_t$, implying that $Q{\,\in\,}\hyperlink{ms}{\M^\H_S}$; hence  \(\hyperlink{msF}{\M_S}{\,\subseteq\,} \hyperlink{ms}{\M^\H_S}\), and the two classes coincide.\smallskip

\textit{Consequently, we shall henceforth write  \(\hyperlink{ms}{\M_S}\) in place of \(\hyperlink{ms}{\M^\H_S}\) for simplicity, whenever no confusion arises.}\smallskip 
		
\noindent \textbf{(d)} We note that the conclusions of Theorem~\ref{thm!}(i) continue to hold if the probability space $(\vO,\vS,P)$ constructed above is replaced by an arbitrary probability space $(\vY,T,\mu)$. Moreover, Theorem~\ref{thm!} remains valid if the classes \(\hyperlink{ms}{\M_S}\), \(\hyperlink{fp}{\F_{P}}\) and  \(\hyperlink{gp}{\G_{P}}\) are replaced by their subclasses
\[
\hypertarget{ms1}{\M^1_S}:=\big\{Q\in\hyperlink{ms}{\M_S} : \E_Q[X_1]<\infty \text{ and } \E_Q[W_1]<\infty\big\},
\]
\[
\hypertarget{fp1}{\F^1_{P}}:=\big\{\ga\in \hyperlink{fp}{\F_{P}} : \E_{P}\big[X_1\cdot e^{\ga(X_{1})}\big]{\,<\,}\infty\big\},
\]
and 
\[
\hypertarget{gp1}{\G^1_{P}}:=\big\{\delta\in\hyperlink{gp}{\G_{P}} : \E_{P}\big[W_1\cdot e^{\delta(W_{1})}\big]{\,<\,}\infty\big\},
\]
respectively.  \smallskip

\noindent \textbf{(e)} Recall that  \(N\) is a \textit{delayed} (or \textit{modified}) renewal process under $P$ (\(P\)-DRP for short), if the sequence $W$ of the interarrival times is $P$-independent and the sequence $\{W_n\}_{n\geq 2}$ is $P$-identically distributed (cf., e.g., Cox~\cite{coxr}, p.~28). To avoid possible confusion, we use the notation $V_0$ for the first interarrival time $W_1$ (recall that $V_0{\,=\,}T_{N_0+1}-0{\,=\,}T_1{\,=\,}W_1$). The characterization given in Theorem~\ref{thm!} extends naturally to the class of DRPs. Note, however, that the corresponding \((P,\F)\)-martingale takes a more complicated  form.\smallskip 

Fix an arbitrary \(t{\,\in\,}\R_+\) and let \(S\) be a \(P\)-compound DRP.  If \(Q\) is a probability measure on \(\vS\) such that \(Q{\uph}\H_t{\,\sim\,}P{\uph}\H_t\) and  \(S\) is a \(Q\)-compound DRP, then by applying arguments similar to those used in the proof of Theorem~\ref{thm!}(i), there exists an essentially unique triplet \((\ga,\delta_1,\delta_2)\) of \(\B(0,\infty)\)-measurable functions, satisfying \(\ga{\,\in\,}\hyperlink{fp}{\F_P}\) and the conditions
\begin{gather}
\E_P\big[e^{\delta_1(V_0)}\big]=\E_P\big[e^{\delta_2(W_2)}\big]=1
\label{drp}
\end{gather}
such that 
\begin{gather*}
\frac{dQ}{dP}{\uph}\H_t=e^{S_t^{(\ga)}+\delta_1(V_0)+\sum_{k=2}^{N_t+1}\delta_2(W_k)}\quad P{\uph}\H_t\text{-a.s.} 
\label{dmart}
\end{gather*}
and 
\begin{gather}
\frac{dQ}{dP}{\uph}\F_t=\1_{\{N_t=0\}}\cdot\frac{\int_{t}^\infty e^{\delta_1(v)}\,P_{V_0}(dv)}{\int_{t}^\infty \,P_{V_0}(dv)}+\1_{\{N_t\geq1\}}\cdot e^{S_t^{(\ga)}+\delta_1(V_0)+\sum_{k=2}^{N_t}\delta_2(W_k)}\cdot\frac{\int_{J_t}^\infty e^{\delta_2(w)}\,P_{W_2}(dw)}{\int_{J_t}^\infty \,P_{W_2}(dw)} \quad P{\uph}\F_t\text{-a.s.}.
\label{dmart1}
\end{gather}

Conversely, by the arguments used in the proof of Theorem~\ref{thm!}(ii),  for every triplet \((\ga,\delta_1,\delta_2)\) of \(\B(0,\infty)\)-measurable functions satisfying  \(\ga{\,\in\,}\hyperlink{fp}{\F_P}\) and \eqref{drp}, condition
\[
Q(A)=\E_P\Big[\1_A\cdot e^{S_t^{(\ga)}+\delta_1(V_0)+\sum_{k=2}^{N_t+1}\delta_2(W_k)}\Big]\,\,\text{ for every }  A \in \H_t
\]
determines a unique probability measure \(Q\) on \(\vS\) such that \(Q{\uph}\H_t{\,\sim\,}P{\uph}\H_t\), \(S\) is a \(Q\)-compound DRP and condition \eqref{dmart1} holds.
\end{rems}

In the next example, we show how starting from a \(P\)-CRP, we can choose a suitable pair of functions \((\ga,\delta)\) leading to the construction of a measure \(Q{\,\in\,}\hyperlink{ms}{\M_S}\) such that the aggregate claims process \(S\) becomes a \(Q\)-CRP with prescribed claim size and interarrival time distributions. In order to present it, denote by \(R_Z\) the \textit{range} of a real-valued random variable \(Z\).
 
\begin{ex}\normalfont 
Assume that \(R_{X_1}{\,=\,}R_{W_1}{\,=\,}(0,\infty)\) and that \(P_{X_1}, P_{W_1}{\,\ll\,}\lambda\), where  $\lambda$  denotes the  Lebesgue measure on \(\B(0,\infty)\). Denote by \(f_{X_1}\) and \(f_{W_1}\) the corresponding density functions. We aim to construct a measure \(Q\) under which both the claim sizes and interarrival times are exponentially distributed with unit rate.\smallskip 

 Consider the pair \((\ga,\delta)\) of real-valued \(\B(0,\infty)\)-measurable functions defined by 
\[
\ga(x):=\ln\frac{e^{-x}}{f_{X_1}(x)}=-x-\ln f_{X_1}(x)\,\,\text{ for all } x\in(0,\infty),
\]
and 
\[
\delta(w):=\ln\frac{e^{-w}}{f_{W_1}(w)}=-w-\ln f_{W_1}(w)\,\,\text{ for every } w\in(0,\infty),
\]
respectively. Then 
\[
\E_P\big[e^{\ga(X_1)}\big]=\int_0^\infty \frac{e^{-x}}{f_{X_1}(x)}\,P_{X_1}(dx)=\int_0^\infty e^{-x}\,\lambda(dx)=1,
\]
and similarly \(\E_P[e^{\delta(W_1)}]{\,=\,}1\); hence $(\ga,\delta){\,\in\,}\hyperlink{fp}{\F_P}{\,\times\,}\hyperlink{gp}{\G_P}$.   Thus, we may apply  Theorem~\ref{thm!}(ii) to obtain  a unique probability measure $Q{\,:=\,}Q^{(\ga,\delta)}{\,\in\,}\hyperlink{ms}{\M_S}$ determined by 
\[ 
Q(A)=\E_P\bigg[\1_A\cdot \Big(\prod_{j=1}^{N_t}\frac{e^{-X_j}}{f_{X_1}(X_j)}\Big)\cdot\Big(\prod_{j=1}^{N_t+1}\frac{e^{-W_j}}{f_{W_1}(W_j)}\Big)\bigg] \,\,\text{ for every } t \in \R_+\text{ and } A \in \H_t.
\]
By Remark~\ref{remp}(a),  the claim size and interarrival times distribution under \(Q\) are 
\[
Q_{X_1}(B)=\E_P\big[\1_{X_1^{-1}[B]}\cdot e^{\ga(X_1)}\big]=\int_B e^{-x}\,\lambda(dx)\,\,\text{ for all } B\in\B(0,\infty),
\]
and 
\[
Q_{W_1}(E)=\E_P\big[\1_{W_1^{-1}[E]}\cdot e^{\delta(W_1)}\big]=\int_E e^{-w}\,\lambda(dw)\,\,\text{ for every } E\in\B(0,\infty),
\]
respectively; hence, \(Q_{W_1}{\,=\,}Q_{X_1}{\,=\,}{\bf Exp}(1)\),  and thus \(S\) is a \(Q\)-compound Poisson process (CPP).\smallskip 

More generally, if \(R_{X_1}, R_{W_1}{\,\subseteq\,}(0,\infty)\), then replacing \(e^{-x}\) and \(e^{-w}\) by arbitrary target density functions \(g\) and \(h\) on \(R_{X_1}\) and \(R_{W_1}\), respectively,  yields a transformation toward any prescribed absolutely continuous distribution. Finally, note that one can work analogously in the discrete case, by replacing densities with probability mass functions on the corresponding supports. 
\end{ex}

The most common change of measures techniques in the Sparre Andersen risk model involve the markovization (either backward or forward) of the reserve process \(R^u\) together with the use of the complex machinery of PDMP  theory (see Dassios \& Embrechts~\cite{daem}, Section~2.3, and Dassios~\cite{das}, Section~3.5, respectively).  A first consequence of Theorem~\ref{thm!} is Corollary~\ref{fmar}, which provides an elementary and unifying construction of the aforementioned martingales. By deriving these martingales directly as Radon-Nikod\'{y}m derivatives, we establish a framework that bypasses the analytical complexity of generator-based methods, see also Proposition 4.15 in Tzaninis \& Macheras~\cite{mt2}.\smallskip

For a random variable $Z$ on $(\vO,\vS)$, denote by $M_{P_Z}$ the  mgf  of the distribution $P_Z$ (cf., e.g., Schmidt~\cite{Sc}, p.~174). Whenever the underlying probability measure is clear from the context, we simply write $M_Z$ instead of $M_{P_Z}$. \smallskip

Let \(c{\,\in\,}\R_+\) and put \(r_{X_1}{\,:=\,}\sup\big\{r{\,\in\,}\R_+ {\,:\,} M_{X_1}(r){\,<\,}\infty\big\}\). Assume that \(r_{X_1}{\,>\,}0\). For every \(r{\,\in\,}[0,r_{X_1})\), let \(\theta(r)\) be the unique solution to the equation
\begin{gather}
	M_{X_1}(r)\cdot M_{W_1}\big(-\theta(r)-c\cdot r\big)=1,
	\label{ad}
\end{gather}
cf., e.g., Rolski et al~\cite{rss}, Lemma~11.5.1(i). Put \(y(r){\,:=\,}\theta(r)+c \cdot r\) for all \(r{\,\in\,}[0,r_{X_1})\).

\begin{cor}
\label{fmar}
Let \(u{\,\in\,}\R_+\) and \(r{\,\in\,}[0,r_{X_1})\), where \(r_{X_1}{\,>\,}0\). The following assertions hold:
\begin{enumerate}
\item for every $Q^{(r)}{\,\in\,}\hyperlink{ms}{\M_S}$ with  	
\[
Q_{X_1}^{(r)}(F):=\frac{\int_F e^{r\cdot x}\, P_{X_1}(dx)}{M_{X_1}(r)}\,\, \text{ and } \,\, Q_{W_1}^{(r)}(E):=\frac{\int_E e^{-y(r)\cdot w }\, P_{W_1}(dw)}{M_{W_1}\big(-y(r)\big)}
\]
for any $F,E{\,\in\,}\B(0,\infty)$, there exists an essentially unique pair $(\ga_r,\delta_r){\,\in\,}\hyperlink{fp}{\F_P}{\,\times\,}\hyperlink{gp}{\G_P}$ with  
\[
\gamma_r(x)=r\cdot x-\ln M_{X_1}(r)\,\,\text{ for all } x\in(0,\infty)
\] 
and  
\[
\delta_r(w)=-y(r)\cdot w-\ln M_{W_1}\big(-y(r)\big)\,\text{ for each } w\in(0,\infty),
\]
such that the conditions 
\begin{gather*}
\frac{dQ^{(r)}}{dP}{\uph}\H_t=M_{X_1}(r)\cdot e^{-r\cdot(R^u_t-u) -y(r)\cdot V_{t} -\theta(r)\cdot t}=:L_t^{(r)}\quad P{\uph}\H_t\text{-a.s.} 
\label{sch1}
\end{gather*}
and 
\begin{gather}
\frac{dQ^{(r)}}{dP}{\uph}\F_t=M_{X_1}(r)\cdot   e^{-r\cdot(R^u_t-u) +y(r)\cdot J_t -\theta(r)\cdot t} \cdot \frac{\int^{\infty}_{J_t} e^{-y(r)\cdot w} \, P_{W_1}(dw)}{{\int_{J_t}^\infty \,P_{W_1}(dw)}}
\quad P{\uph}\F_t\text{-a.s.} 
\label{sch2}
\end{gather}
hold for all $t{\,\in\,}\R_+$;
\item 
conversely, for any pair $(\ga_r,\delta_r)$ of real-valued $\B(0,\infty)$-measurable functions as in (i), there exists a unique probability measure $Q^{(r)}{\,\in\,}\hyperlink{ms}{\M_S}$ determined by \eqref{mq}, namely  
\[
Q^{(r)}(A)=\E_P\big[\1_A\cdot L_t^{(r)}\big]\,\,\text{ for all } t\in\R_+\text{ and } A\in\H_t, 
\]
and satisfying condition \eqref{sch2}.
\end{enumerate}
\end{cor}
	
\noindent The proof follows immediately from  Theorem~\ref{thm!} and therefore it is omitted.
	
\begin{rems}
\label{mark}
\normalfont
{\bf (a)}  The $(P,\H)$-martingale $L^{(r)}{\,:=\,}\{L_t^{(r)}\}_{t\in\R_+}$ in Corollary~\ref{fmar} coincides with the one appearing in Dassios~\cite{das}, Theorem~1.1.3. Its construction there requires the forward markovization of $R^u$, obtained by considering the bivariate Markov process $\{(R^u_t,V_t)\}_{t\in\R_+}$ and applying either the theory of PDMPs (see Dassios~\cite{das}, Theorem~1.1.3) or the theory of MAPs (cf., e.g., Asmussen \& Albrecher~\cite{asal}, pp.~159-160). In contrast, the present construction relies only on elementary probabilistic arguments.\smallskip

\noindent	{\bf (b)} In the special case $P_{W_1}{\,\ll\,}\lambda$, the $(P,\F)$-martingale involved in condition~\eqref{sch2} coincides with the one presented  in Dassios \& Embrechts~\cite{daem}, Theorem~10, which is the martingale arising  from the backward markovization of a reserve process. Note, however, that in Corollary~\ref{fmar}, the assumption $P_{W_1}{\,\ll\,}\lambda$ is not required for the construction.
\end{rems}
	
	
\section{Change of measures and ruin probabilities}\label{app}
 
Classically, the martingales $L^{(r)}$ and $M^{(r)}$ appearing in Corollary~\ref{fmar} are used either to derive standard results for the ruin probability $\psi(u)$ (cf., e.g., Schmidli~\cite{scm}, Sections~8.2 and~8.3) or to construct  IS schemes (cf., e.g., Asmussen \& Albrecher~\cite{asal}, Chapter~XV, Section~3) in the Cram\'er–Lundberg and Sparre Andersen risk models. However, these approaches fundamentally rely on the existence of an adjustment coefficient (cf., e.g., Schmidli~\cite{scm}, p.~133), and therefore exclude heavy-tailed distributions. Since both martingales arise as special instances of the general martingales $L^{(\ga,\delta)}$ and $M^{(\ga,\delta)}$ introduced in Theorem~\ref{thm!}, it is natural to ask whether suitable subclasses of  $\hyperlink{fp}{\F_P}{\,\times\,} \hyperlink{gp}{\G_P}$ and $\hyperlink{ms}{\M_S}$ can be identified for applications to the ruin problem.\smallskip 

In this section we provide a complete characterization of all probability measures  \(Q{\,\in\,}\hyperlink{ms}{\M_S}\) under which ruin occurs \(Q\)-a.s.. This characterization yields an explicit representation of the ruin probability $\psi(u)$ that is suitable for simulation purposes, including settings with heavy-tailed distributions where classical exponential tilting fails.  \smallskip 

\textit{Throughout what follows, we assume that \(P{\,\in\,}\hyperlink{ms1}{\M^1_S}\) and that \(c\) is a positive real constant, satisfying the net profit condition}
	\begin{gather}
		c\cdot \E_P[W_1]>\E_P[X_1]
		\tag{NPC}
		\label{npc}
	\end{gather}
\textit{since otherwise ruin occurs \(P\)-a.s.. 
}
\smallskip 

In order to formulate the desired characterization, we introduce the classes
\[
\hypertarget{rms}{\mathcal{R}_S} :=\Big\{Q\in\hyperlink{ms1}{\M^1_S} : Q\big(\{\tau_u<\infty\}\big)=1\Big\}
\]
and 
\[
\hypertarget{rp}{\C}_P:=\Big\{(\ga,\delta)\in\hyperlink{fp1}{\F^1_P}{\,\times\,} \hyperlink{gp1}{\G^1_P} : c\cdot \E_P\big[W_1\cdot e^{\delta(W_1)}\big]\leq\E_P\big[X_1\cdot e^{\ga(X_1)}\big]\Big\}.
\]

 \begin{rem}
	\normalfont
	\label{prem}
\textbf{(a)} The classes \(\hyperlink{rp}{\C_P}\) and \(\hyperlink{rms}{\mathcal{R}_S}\) are proper subclasses of \(\hyperlink{fp1}{\F^1_P}{\,\times\,} \hyperlink{gp1}{\G^1_P}\) and \(\hyperlink{ms1}{\M^1_S}\), respectively.\smallskip
	
In fact,  assume that \(r_{W_1}{\,:=\,}\sup\big\{r{\,\in\,}\R_+ {\,:\,} M_{W_1}(r){\,<\,}\infty\big\}{\,>\,}0\), and that \(\E_P[X_1^2]{\,<\,}\infty\). For a fixed \(r{\,\in\,}(0,r_{W_1})\) define the \(\B(0,\infty)\)-measurable functions  \(\ga_r,\delta_r\)  by 
\[
\ga_r(x){\,:=\,}-r\cdot x-\ln M_{X_1}(-r)\text{ for every } x{\,\in\,}(0,\infty)
\]
 and 
\[
\delta_r(w){\,:=\,} r\cdot w-\ln M_{W_1}(r)\text{ for all }w{\,\in\,}(0,\infty),
\]
respectively. A direct computation shows that
 \((\ga_r,\delta_r){\,\in\,} \hyperlink{fp1}{\F^1_P}{\,\times\,} \hyperlink{gp1}{\G^1_P}\). Thus, by Theorem~\ref{thm!}(ii) and Remark~\ref{remp}(d), there exists a unique probability measure \(Q^{(r)}{\,:=\,}Q^{(\ga_r,\delta_r)}{\,\in\,}\hyperlink{ms1}{\M^1_S}\) determined by \eqref{mq} and satisfying \eqref{mart1}.  \smallskip	
 
Next, define the functions \(h_1,h_2{\,:\,}[0,r_{W_1}){\,\to\,}\R\)   by   
\[
h_1(r){\,:=\,}-\ln M_{X_1}(-r)\text{ and } h_2(r){\,:=\,}\ln M_{{W}_1}(r)\text{ for all } r{\,\in\,}[0,r_{W_1}),
\]
 respectively. Then  \(h_1^{\prime\prime}(r){\,<\,}0\) and \(h_2^{\prime\prime}(r){\,>\,}0\) for all \(r{\,\in\,}[0,r_{W_1})\), implying that  \(h_1\) is strictly concave and \(h_2\) is strictly convex; hence,  \(h_1^\prime\)  is strictly decreasing and \(h_2^\prime\) is strictly increasing on \([0,r_{W_1})\). Therefore, 
\[
\E_P[X_1] = h_1^\prime(0) > h_1^\prime(r)= \E_P\big[X_1\cdot e^{\ga_r(X_1)}\big] \text{ and } \E_P[W_1] = h_2^\prime(0) < h_2^\prime(r)=\E_P\big[W_1\cdot e^{\delta_r(W_1)}\big]
 \]
 for all \(r{\,\in\,}(0,r_{W_1})\). This, along with \eqref{npc}, yields 
\[
\frac{\E_P\big[X_1\cdot e^{\ga_r(X_1)}\big]}{\E_P\big[W_1\cdot e^{\delta_r(W_1)}\big]}<\frac{\E_P[X_1]}{\E_P[W_1]}<c\Leftrightarrow  \frac{\E_{Q^{(r)}}[X_1]}{\E_{Q^{(r)}}[W_1]}<\frac{\E_P[X_1]}{\E_P[W_1]}<c,
\]
for all \(r{\,\in\,}(0,r_{W_1})\); hence,  \((\ga_r,\delta_r){\,\notin\,}\hyperlink{rp}{\C_P}\) and \(Q^{(r)}{\,\notin\,}\hyperlink{rms}{\mathcal{R}_S}\), which proves that
\[
\hyperlink{rp}{\C_P}\subsetneqq \hyperlink{fp1}{\F^1_P} \times \hyperlink{gp1}{\G^1_P}\text{ and } \hyperlink{rms}{\mathcal{R}_S} \subsetneqq \hyperlink{ms1}{\M^1_S}.
 \]
	
\noindent\textbf{(b)}  Let \((\ga,\delta){\,\in\,}\hyperlink{fp1}{\F^1_P}{\,\times\,} \hyperlink{gp1}{\G^1_P}\).
If \(\ga\) is increasing and \(\delta\) is decreasing, then, applying Schmidt~\cite{sc14}, Theorem~2.2, we obtain
\[
\E_{P}\big[X_1 \cdot e^{\ga(X_1)}\big]\geq\E_{P}[X_1]\text{ and } \E_{P}\big[W_1\cdot  e^{\delta(W_1)}\big]\leq \E_{P}[W_1].
\]
This implies that the pair \((\ga,\delta)\) is an element of \(\hyperlink{rp}{\C_P}\).

\end{rem}	
	
The following result, which is a direct consequence of Theorem~\ref{thm!}, provides a complete characterization of the relevant objects and establishes a one-to-one correspondence between \(\hyperlink{rp}{\C_P}\) and \(\hyperlink{rms}{\mathcal{R}_S}\). In doing so, it extends the classical Esscher-based representation of the ruin probability to settings beyond the light-tailed case.
 
\begin{thm}\label{gform}  
The following assertions hold:
\begin{enumerate}
\item
for every  \(Q{\,\in\,}\hyperlink{rms}{\mathcal{R}_S}\) there exists an essentially unique  \((\ga,\delta){\,\in\,}\hyperlink{rp}{\C_P}\) satisfying conditions \eqref{mart} and \eqref{mart1}. In particular, the ruin probability admits the representation
\begin{gather}
\psi(u)=\E_Q\Big[e^{-S_{\tau_u}^{(\ga)}-T^{(\delta)}_{N_{\tau_u}}}\Big]\,\,\text{ for all } u\in\R_+;
\label{ruin}
\tag{$\text{ruin}(P)$}
\end{gather} 
\item
conversely, for every \((\ga,\delta){\,\in\,}\hyperlink{rp}{\C_P}\), condition \eqref{mq} determines a unique probability measure $Q{\,:=\,}Q^{(\ga,\delta)}{\,\in\,}\hyperlink{rms}{\mathcal{R}_S}$ satisfying  \eqref{mart1}. In particular, the ruin probability admits the representation \eqref{ruin}.
\end{enumerate}
\end{thm}

\begin{proof}
Let \(u,t{\,\in\,}\R_+\) be arbitrary but fixed. \smallskip 

\noindent Ad (i): Let \(Q{\,\in\,}\hyperlink{rms}{\mathcal{R}_S}\). As \(\hyperlink{rms}{\mathcal{R}_S}{\,\subsetneqq\,}\hyperlink{ms1}{\M^1_S}\) by Remark~\ref{prem}(a), we can apply Theorem~\ref{thm!}(i) together with Remark~\ref{remp}(d) to get an essentially unique pair \((\ga,\delta){\,\in\,}\hyperlink{fp1}{\F^1_P} \times \hyperlink{gp1}{\G^1_P}\) such that conditions \eqref{mart} and \eqref{mart1} hold true. But since \(Q\big(\{\tau_u{\,<\,}\infty\}\big){\,=\,}1\) we have \(c\cdot\E_Q[W_1] {\,\leq\,}\E_Q[X_1]\) \adc{(cf., e.g., Schmidli~\cite{scm}, p.~132)}. This, along with Remark~\ref{remp}(a), yields \( c\cdot \E_P\big[W_1\cdot e^{\delta(W_1)}\big]{\,\leq\,}\E_P\big[X_1\cdot e^{\ga(X_1)}\big]\), implying that  \((\ga,\delta)\) is an element of \(\hyperlink{rp}{\C_P}\). We split the remainder of the proof into the following two steps.\smallskip 

\noindent \textbf{(a)} The following condition holds     
\[
\E_Q\big[e^{-\delta(W_{N_{t}+1})}\mid\F_{t}\big]=\frac{\int_{J_{t}}^\infty \,P_{W_1}(dw)}{\int_{J_{t}}^\infty e^{\delta(w)}\,P_{W_1}(dw)}\quad Q{\uph}\F_{t}\text{-a.s.}.
\]

In fact, since the processes $L^{(\ga,\delta)}$ and $M^{(\ga,\delta)}$, involved in conditions \eqref{mart} and \eqref{mart1}, are a.s. positive $(P,\H)$- and $(P,\F)$-martingales, respectively, and $P{\uph}\H_t{\,\sim\,}Q{\uph}\H_t$, we get that the families $1/L^{(\ga,\delta)}$ and $1/M^{(\ga,\delta)}$ are a.s. positive $(Q,\H)$- and $(Q,\F)$-martingales, respectively. Thus, for every \(A{\,\in\,}\F_t\) we get 
\[
\int_A 1/L_t^{(\ga,\delta)}\,dQ=\int_A \E_Q\big[1/L_t^{(\ga,\delta)}\mid\F_t\big]\,dQ=\int_A 1/M_t^{(\ga,\delta)}\,dQ,
\]
implying, together with the \(\F_t\)-measurability of \(S_{t}^{(\ga)}\) and \(T^{(\delta)}_{N_{t}}\), that 
\[
\int_A e^{-S_{t}^{(\ga)}-T^{(\delta)}_{N_{t}}}\cdot\Big(\E_Q\big[e^{-\delta(W_{N_{t}+1})}\mid\F_{t}\big]-\frac{\int_{J_{t}}^\infty \,P_{W_1}(dw)}{\int_{J_{t}}^\infty e^{\delta(w)}\,P_{W_1}(dw)}\Big)\,dQ=0;
\]
hence 
\[
\E_Q\big[e^{-\delta(W_{N_{t}+1})}\mid\F_{t}\big]=\frac{\int_{J_{t}}^\infty \,P_{W_1}(dw)}{\int_{J_{t}}^\infty e^{\delta(w)}\,P_{W_1}(dw)}\quad Q{\uph}\F_{t}\text{-a.s.}.
\]
This completes the proof of step (a).\smallskip

\noindent \textbf{(b)} Condition \eqref{ruin} holds true.\smallskip 

 Since \(1/L^{(\ga,\delta)}\) is an a.s. positive \((Q,\H)\)-martingale, and  \(\F\) is right-continuous (cf., e.g., Jacod~\cite{jacod}, Proposition~3.39), it follows by e.g.,  Rolski et al.~\cite{rss}, Theorem~10.1.1, that $\tau_u$ is an $\F$-stopping time (and thus an $\H$-stopping time, since $\F_t{\,\subseteq\,}\H_t$); hence, we can apply e.g., Rolski et al.~\cite{rss}, Lemma~10.2.2(b), to get  
\[
\psi(u,y)=\E_Q\Big[\1_{\{\tau_u\leq y\}}\cdot 1/L^{(\ga,\delta)}_{\tau_u}\Big] 
\]
for every real number $y{\,>\,}0$. Letting now $y{\,\uparrow\,}\infty$, we get by the Monotone Convergence Theorem that
\begin{gather*}
	\psi(u)=\E_{Q}\Big[\1_{\{\tau_u<\infty\}}\cdot 1/L^{(\ga,\delta)}_{\tau_u}\Big],
	\label{ruinfor}
\end{gather*} 
implying, together with the equality \(Q\big(\{\tau_u{\,<\,}\infty\}\big){\,=\,}1\),  that 
\begin{gather}
	\psi(u)=\E_Q\Big[1/L^{(\ga,\delta)}_{\tau_u}\Big]=\E_Q\Big[ e^{-S_{\tau_u}^{(\ga)}-T^{(\delta)}_{N_{\tau_u}}} \cdot  \E_Q\big[e^{-\delta(W_{N_{\tau_u}+1})}\mid\F_{\tau_u}\big]\Big].
	\label{ruin321}
\end{gather}
Using now step (a), in conjunction with the Optional Stopping Theorem (cf., e.g., Schmidli~\cite{scm}, Proposition~B.2), we get
\[
\E_Q\big[e^{-\delta(W_{N_{\tau_u}+1})}\mid\F_{\tau_u}\big]=\frac{\int_{J_{\tau_u}}^\infty \,P_{W_1}(dw)}{\int_{J_{\tau_u}}^\infty e^{\delta(w)}\,P_{W_1}(dw)}=1\quad Q{\uph}\F_{\tau_u}\text{-a.s.,}
\]
where the second equality follows since \(J_{\tau_u}{\,=\,}0\) and \(\delta{\,\in\,}\hyperlink{gp}{\G_P}\). This, together with condition~\eqref{ruin321},  proves \eqref{ruin}, completing the proof of (b) and the whole proof of statement (i).\smallskip  

\noindent Ad (ii): Let \((\ga,\delta){\,\in\,}\hyperlink{rp}{\C_P}\). As \(\hyperlink{rp}{\C_P}{\,\subsetneqq\,}\hyperlink{fp1}{\F^1_P} \times \hyperlink{gp1}{\G^1_P}\) by Remark~\ref{prem}(a), we can apply Theorem~\ref{thm!}(ii) together with Remark~\ref{remp}(d) to obtain a unique probability measure $Q{\,:=\,}Q^{(\ga,\delta)}{\,\in\,}\hyperlink{ms1}{\M^1_S}$ determined by condition \eqref{mq}  and satisfying \eqref{mart1}. But since  \((\ga,\delta){\,\in\,}\hyperlink{rp}{\C_P}\), we get  \( c\cdot \E_P\big[W_1\cdot e^{\delta(W_1)}\big]{\,\leq\,}\E_P\big[X_1\cdot e^{\ga(X_1)}\big]\), implying, together with Remark~\ref{remp}(a), that \(c\cdot \E_Q[W_1]\leq\E_Q[X_1]\). Thus, \(Q\big(\{\tau_u{\,<\,}\infty\}\big){\,=\,}1\) (cf., e.g., Schmidt~\cite{Sc}, Corollary 7.1.4); hence, \(Q{\,\in\,}\hyperlink{rms}{\mathcal{R}_S}\). The proof of condition \eqref{ruin} follows by the arguments used in the proof of assertion (i). This completes the proof of statement (ii) and the whole proof.  
\end{proof}

\begin{rem}\label{pemmcl}\normalfont
Let $Y{\,:=\,}\{Y_t\}_{t\in\R_+}$ be a stochastic process on $(\vO,\vS)$ that is adapted to a filtration $\G{\,:=\,}\{\G_t\}_{t\in\R_+}$. A probability measure $Q$ on $\vS$  is called a  \textit{\(\G\)-progressively equivalent martingale measure} (\(\G\)-PEMM) for $Y$ if $Q{\uph}\G_t{\,\sim\,}P{\uph}\G_t$ for all \(t{\,\in\,}\R_+\) and $Y$ is a  $(Q, \G)$-martingale.\smallskip

The use of PEMMs in non-life insurance mathematics was proposed by Delbaen \& Haezendonck in their seminal paper \cite{dh} in an effort to  \enquote{create a mathematical framework to deal with finance related to risk processes}  within the context of CPPs, see also M\o{}ller~\cite{moller04}, Section 4.  The main result in Delbaen \& Haezendonck~\cite{dh} is a characterization of all progressively equivalent probability measures  $Q$ on the domain of $P$ such that a \(P\)-CPP remains a \(Q\)-CPP, see \cite{dh}, Proposition 2.2. Lyberopoulos \& Macheras~\cite{lm3} generalized this characterization for mixed compound Poisson processes (see \cite{lm3}, Theorem 5.3), while Macheras \& Tzaninis~\cite{mt3} and Tzaninis \& Macheras~\cite{mt2} extended the result of  Delbaen \& Haezendonck by characterizing all progressively equivalent probability measures  $Q$ on the domain of $P$  that convert a given \(P\)-compound (mixed) renewal process into a \(Q\)-compound (mixed) Poisson one (see \cite{mt3}, Corollary 3.1, and \cite{mt2}, Corollary 4.8).\smallskip 
		
In particular, \(\F\)-PEMMs appear to be natural candidates for obtaining ruin-inducing probability. Denote by \(\hypertarget{martc}{\mathbb{M}^\F_S}\) the class of all \(\F\)-PEMMs for the price process \(\{u-R^u_t\}_{t\in\R_+}\). Since \(u-R^u_t{\,=\,}S_t-c\cdot t\) for all \(t{\,\in\,}\R_+\), it follows from Ikeda \& Watanabe~\cite{ikwa89},  Theorem II.6.2, that
		\[
		\hyperlink{martc}{\mathbb{M}^\F_S}=\big\{Q\in\hyperlink{ms1}{\M^1_S} : S \text{ is a } Q\text{-CPP and } c\cdot\E_Q[W_1]{\,=\,}\E_Q[X_1]\big\}\subsetneqq \hyperlink{rms}{\mathcal{R}_S},
		\]
		where the strict inclusion follows since the equality \(c\cdot\E_Q[W_1]{\,=\,}\E_Q[X_1]\) implies that \(Q\big(\{\tau_u{\,<\,}\infty\}\big){\,=\,}1\) (see, e.g., Schmidt~\cite{Sc}, Corollary 7.1.4). 
\end{rem}
 
Theorem~\ref{gform} provides a martingale representation of the ruin probability \(\psi(u)\) by introducing a new probability measure \(Q\) that preserves the structure of the given CRP and under which ruin occurs \(Q\)-a.s.. Such a measure is obtained by modifying the distribution of  claim sizes (through the function \(\ga\)) and/or the distribution of   claim interarrival times  (through the function \(\delta\)). However, for practical purposes, e.g. for an IS simulation, representation~\eqref{ruin} is still too general and must be adapted to a discrete time framework. Define 
\[
G_n:=X_n-c\cdot W_n\text{ for every } n\in\N. 
\] 
and 
\[
Z_n:=\begin{cases}
0, & \text{ for } n=0,\\
\sum_{k=1}^n G_k, & \text{ for } n\in\N.
\end{cases}
\]
Since \(N\) has zero probability of explosion, it follows by, e.g., Schmidt~\cite{Sc}, Lemma~7.1.2, that 
\[
\psi(u)=P\big(\{\sup_{n\in\N_0} Z_n\geq u\}\big).
\]
We now present a simple IS algorithm for estimating \(\psi(u)\), based on  condition \eqref{ruin} and the discretization above. Note that the algorithm always converges since ruin occurs \(Q\)-a.s..\smallskip 
 
 \begin{myblock}[label=alg1]{Importance Sampling  estimation of $\psi(u)$ using Theorem~\ref{gform}}\smallskip 
    \label{alg:IS_ruin}
		\begin{enumerate}
			\item[\textit{Inputs:}] 
			\begin{enumerate}
				\item[\(\bullet\)] The distributions \(P_{X_1}\) and \(P_{W_1}\);
				\item[\(\bullet\)] the initial reserve \(u{\,\in\,}\R_+\);
				\item[\(\bullet\)] the premium rate \(c{\,\in\,}\big( \E_P[X_1]/\E_P[W_1],\infty)\);
				\item[\(\bullet\)] a pair \((\ga,\delta){\,\in\,}\hyperlink{rp}{\C_P}\);
				\item[\(\bullet\)] the number of simulations \(K\).
			\end{enumerate}
			\item[\textit{Output:}] An estimate \(\Psi_K(u)\)  of \(\psi(u)\).
		\end{enumerate}
		
		\noindent \textbf{Step 1.} Obtain the probability distributions
		\[
		Q^{(\ga,\delta)}_{X_1}(dx)=e^{\ga(x)}\cdot P_{X_1}(dx)\,\,\text{ and }\,\,  Q^{(\ga,\delta)}_{W_1}(dw)=e^{\delta(w)}\cdot P_{W_1}(dw)
		\]
		\noindent \textbf{Step 2.} For \(i = 1, \ldots, K\):
		\begin{enumerate}
			\item[(a)] Set \(Z^{(i)}_0 \leftarrow 0\). 
			\item[(b)] For \(n\in\N\), repeat the loop:
			\begin{itemize}
				\item[\(\diamond\)] Sample \(W^{(i)}_n\) from \(Q^{(\gamma,\delta)}_{W_1}\) and \(X^{(i)}_n\) from \(Q^{(\gamma,\delta)}_{X_1}\).
				\item[\(\diamond\)] Update \(Z^{(i)}_n \leftarrow Z^{(i)}_{n-1} + X^{(i)}_n - c \cdot W^{(i)}_n\).
				\item[\(\diamond\)] Terminate the loop if \(Z^{(i)}_n \ge u\).
			\end{itemize} 
			\item[(c)] If  \(Z^{(i)}_n \ge u\), then set \(\tau^{(i)}_u\leftarrow T^{(i)}_{n}=:\sum_{m=1}^n W^{(i)}_m\) and \(N_{\tau_u}^{(i)}:=N_{\tau^{(i)}_u}^{(i)}  \leftarrow n\) and compute
			\[
			\ell^{(\ga,\delta)}_i(u):=\exp\bigg(- \sum_{j=1}^{N_{\tau_u}^{(i)}}\gamma(X^{(i)}_j)
			- \sum_{j=1}^{N_{\tau_u}^{(i)}} \delta(W^{(i)}_j)\bigg).
			\]
		\end{enumerate}
		\noindent \textbf{Step 3.} Return the IS estimate of \(\psi(u)\)
		\[
		\Psi_K(u):=\Psi^{(\ga,\delta)}_K(u):=\frac{1}{K}\cdot\sum_{i=1}^K \ell^{(\ga,\delta)}_{i}(u) .
		\]
	\end{myblock} 
	 
\begin{rems}\label{gerber-shiu}\normalfont
\textbf{(a)} Since their introduction by Gerber \& Shiu~\cite{gesh98}, Gerber-Shiu functions have become a central object in actuarial science, as they provide a unified framework for the analysis of a wide range of ruin-related quantities. For an initial reserve $u{\,\in\,}\R_+$, the \textit{Gerber-Shiu} function is defined as the expected discounted penalty at ruin
\[
\phi_\eta(u)=\E_P\Big[\1_{\{\tau_u<\infty\}}\cdot e^{-\eta\cdot\tau_u}\cdot w\big(R^u_{\tau_u-},\xi(u)\big)\Big]
\]
where $\eta{\,\in\,}\R_+$ is a \textit{discount rate}, $w{\,:\,}(0,\infty)^2 {\,\to\,} \R_+$ is a \textit{penalty function}, $R^u_{\tau_u-}{\,:=\,}\lim_{t\uparrow \tau_u} R^u_t$ denotes the \textit{surplus prior to ruin}, and $\xi(u){\,=\,}|R^u_{\tau_u}|$ is the \textit{deficit at ruin}, see~\cite{gesh98}, pp.~49-50 (see also He et al.~\cite{hksy23} for an overview of subsequent generalizations). Closed-form or semi-closed representations for $\phi_\eta$ are relatively scarce in the literature, even for classical risk models such as the Cram\'{e}r-Lundberg or Sparre Andersen ones. In particular, in the presence of heavy-tailed claim size distributions, most available results are restricted to asymptotic approximations, see, for example, Cheng \& Tang~\cite{chta03} and Tang \& Wei~\cite{taw10}.\smallskip 

The exponential change of measures framework,  developed for \(\psi(u)\) in Theorem~\ref{gform}, extends naturally to Gerber-Shiu functions.   In fact, for a given pair \((\ga,\delta){\,\in\,}\hyperlink{rp}{\C_P}\), it follows from Theorem~\ref{gform}(ii) that  condition~\eqref{mq} determines a unique probability measure \(Q{\,:=\,}Q^{(\ga,\delta)}{\,\in\,}\hyperlink{rms}{\mathcal{R}_S}\), while by applying the arguments  appearing in the proof of formula~\eqref{ruin}, one can show that \(\phi_\eta\) admits the representation
\[
\phi_\eta(u)=\E_Q\Big[e^{-\eta\cdot\tau_u}\cdot w\big(R^u_{\tau_u-},\xi(u)\big)\cdot e^{-S_{\tau_u}^{(\ga)}-T^{(\delta)}_{N_{\tau_u}}}\Big]\,\,\text{ for every } u\in\R_+.
\]
This representation provides a convenient starting point for the numerical estimation of Gerber-Shiu functions via IS simulation.  While the above discussion shows that Gerber-Shiu functions can be treated within our framework, we do not consider them further in this paper. \smallskip 

\noindent \textbf{(b)} The framework developed in Theorem~\ref{gform} can also accommodate  both  finite-time ruin probabilities and  alternative solvency thresholds.
\smallskip 

In fact, since \(P\) and \(Q\) are progressively equivalent on \(\H\), then for all \(y{\,\in\,}(0,\infty)\) and \(u{\,\in\,}\R_+\)  the finite-time ruin probability admits the representation 
\[
\psi(u,y) = P\big(\{\tau_u\leq y\}\big)=\E_Q \Big[\1_{\{\tau_u \leq y\}} \cdot e^{-S_{\tau_u}^{(\ga)}-T^{(\delta)}_{N_{\tau_u}}} \Big].
\]
Accordingly, the IS algorithm presented above must be modified as follows: in Step~2(b) the simulation is terminated either when \(Z^{(i)}_n {\,\geq\,} u\) (ruin) or when \(\tau^{(i)}_u{\, >\,} y\) (time horizon exceeded).   The resulting estimator is
\[
\Psi_K(u,y)=\frac{1}{K}\cdot\sum_{i=1}^K\1_{\{\tau_u^{(i)} \le y\}}\cdot
\exp\bigg(- \sum_{j=1}^{N_{\tau_u}^{(i)}} \gamma(X_j^{(i)})
- \sum_{j=1}^{N_{\tau_u}^{(i)}} \delta(W_j^{(i)})
\bigg).
\]
 
Furthermore, let \(b{\,\in\,}[0,u]\) and define the  stopping time \(\tau_{u,b}{\,:=\,}\inf\{ t{\,\in\,}\R_+ {\,:\,} R_t^u {\,<\,} b \}\).  Since \(R_t^u - b{\,=\,}R_t^{u-b}\), the event \(\{\tau_{u,b} {\,<\,} \infty\}\) coincides with the ruin event for initial capital \(u-b\). Hence, by Theorem~\ref{gform} we get
\[
P\big(\{\tau_{u,b} < \infty\}\big)=\E_Q\Big[
e^{-S_{\tau_{u,b}}^{(\gamma)} - T^{(\delta)}_{N_{\tau_{u,b}}}}\Big]\,\,\text{ for all } u\in\R_+\text{ and } b\in[0,u],
\]
and the IS algorithm applies without other modifications than replacing \(u\) by \(u-b\).\smallskip
 
In particular, regulatory capital requirements such as the \textit{Solvency Capital Requirement} (SCR) under Solvency II can be treated within the same framework, both over finite and infinite regulatory horizons, by combining the spatial shift described above with the finite-time stopping rule when appropriate. 
\end{rems}

\section{Examples}\label{examples}

In this section, we illustrate how Theorem~\ref{gform} can be used to construct ruin-inducing probability measures and   explicit representations of  \(\psi(u)\) for three representative choices of \((\ga,\delta){\,\in\,}\hyperlink{rp}{\C_P}\). These  include, as special cases of our framework, (i) the Esscher transform (see Section~\ref{SALT}), (ii) a linear tilting approach (see Section~\ref{mmm}), and (iii) a hazard-rate twisting approach (see Section~\ref{PH}). This demonstrates that the proposed framework unifies several change of measures techniques that have been developed in different areas in the literature.

\makeatletter
\renewcommand{\subsection}{\@startsection{subsection}{2}{0pt}%
	{-3.25ex plus -1ex minus -.2ex}%
	{1.5ex plus .2ex}%
	{\normalfont\bfseries}} 
\makeatother

\subsection{The Esscher transform}\label{SALT} 

In this section, we show that the classical Esscher transform, widely used in Risk Theory (cf., e.g., Schmidli~\cite{scm}, Sections  8.2 and 8.3), is not a special ad hoc construction, but rather a particular instance of the ruin-inducing class characterized in Theorem~\ref{gform}. Consequently, the traditional adjustment coefficient approach appears as a special case of a broader structural framework.\smallskip 

Assume that \(r_{X_1}{\,>\,}0\). 
For every $r{\,\in\,}[0,r_{X_1})$, consider the pair $(\ga_r,\delta_r)$ of real-valued $\B(0,\infty)$-measurable functions defined by  
\[
\ga_r(x)=r\cdot x-\ln M_{X_1}(r)\,\text{ for all } x\in(0,\infty) 
\]
and
\[
\delta_r(w)=-y(r)\cdot w-\ln M_{W_1}\big(-y(r)\big)\,\text{ for every } w\in(0,\infty), 
\] 
where $y(r){\,=\,}\theta(r){\,+\,}c\cdot r$. A direct computation shows that $(\ga_r,\delta_r){\,\in\,}\hyperlink{fp1}{\F^1_P}{\,\times\,} \hyperlink{gp1}{\G^1_P}$.\smallskip 

Since the function \(\theta\) is infinitely differentiable in a right-neighbourhood of \(0\) (cf., e.g., Schmidli~\cite{scm}, p.~133), differentiation of condition \eqref{ad} with respect to \(r\) yields
\begin{gather}
\theta^\prime(r)=\frac{\E_P\big[X_1\cdot e^{r\cdot X_1}\big]\cdot M_{W_1}\big(-y(r)\big)}{\E_P\big[W_1\cdot e^{-y(r)\cdot W_1}\big]\cdot M_{X_1}(r)}-c=\frac{\E_P\big[X_1\cdot e^{\ga_r(X_1)}\big]}{\E_P\big[W_1\cdot e^{\delta_r(W_1)}\big]}-c.
\label{es1}
\end{gather}
This, together with the~\eqref{npc}, yields \(\theta^\prime(0){\,<\,}0\). Since the function \(\theta\) is strictly convex and \(\theta(0){\,=\,}0\) (cf., e.g.,~\cite{scm}, p.~133), it follows that there exists at most one point \(r^{(m)}{\,\in\,}(0,r_{X_1})\) such that \(\theta^\prime\big(r^{(m)}\big){\,=\,}0\).  When such a point exists, we have \(\theta^\prime(r){\,>\,}0\) for every \(r{\,\in\,}(r^{(m)}, r_{X_1})\), implying,  together with condition \eqref{es1}, that \((\ga_r,\delta_r){\,\in\,}\hyperlink{rp}{\C_P}\) for all \(r{\,\in\,}[r^{(m)}, r_{X_1})\).\smallskip

Assume now that \(r^{(m)}{\,\in\,}(0,r_{X_1})\) and fix an arbitrary \(r{\,\in\,}[r^{(m)}, r_{X_1})\). By Theorem~\ref{gform}(ii), there exists a unique probability measure \(Q^{(r)}{\,:=\,}Q^{(\ga_r,\delta_r)}{\,\in\,}\hyperlink{rms}{\mathcal{R}_S}\) determined by  condition~\eqref{mq}, namely
\begin{gather}
Q^{(r)}(A)=M_{X_1}(r)\cdot\E_P\Big[\1_A\cdot e^{r\cdot S_t-\theta(r)\cdot T_{N_t+1}}\Big]\,\,\text{ for all } t\in\R_+ \text{ and } A\in\H_t,
\label{esmes}
\end{gather}
satisfying condition~\eqref{mart1}. In particular, the ruin probability admits the representation~\eqref{ruin} with 
\begin{gather}
	\psi(u)=\E_{Q^{(r)}}\big[e^{r\cdot R^u_{\tau_u}+\theta(r)\cdot \tau_u}\big]\cdot e^{-r\cdot u}\,\,\text{ for every } u\in\R_+.
	\label{332}
\end{gather}
Furthermore, if the equation \(\theta(r){\,=\,}0\) has a unique positive solution \(\rho{\,\in\,}(r^{(m)},r_{X_1})\), then  choosing \(r{\,=\,}\rho\) in condition~\eqref{332} yields
\[
\psi(u)=\E_{Q^{(\rho)}}\big[e^{\rho\cdot R^u_{\tau_u}}\big]\cdot e^{-\rho\cdot u}\,\,\text{ for all } u\in\R_+.
\]
  
\begin{srem}\label{memmrem}\normalfont
Assume that $P_{W_1}{\,=\,}{\bf Exp}(\beta)$, where $\be{\,\in\,}(0,\infty)$ (cf., e.g., Schmidli~\cite{scm}, p.~229), and denote by \(Q^{(m)}\) the probability measure on \(\vS\) determined by \eqref{esmes} for \(r{\,=\,}r_m\). Let $T{\,<\,}\infty$ be a fixed time horizon. The restriction $Q^{(m)}_T$ of $Q^{(m)}$ to $\F_T$ coincides with the \textit{minimal entropy  martingale measure} (MEMM) for CPPs (see  M\o{}ller~\cite{moller04}, Section 4.3), that is, the equivalent martingale measure (EMM) minimizing the relative entropy over all EMMs $Q_T$ for the price process $\{u-R^u_t\}_{t\in[0,T]}$.  Under the MEMM, the corresponding no-arbitrage premium is
\[
c^{(m)}:=c\big(r^{(m)}\big)=\beta\cdot\E_P\big[X_1\cdot e^{r^{(m)}\cdot X_1}\big].
\]

\end{srem}

\subsection{The linear tilting approach}\label{mmm}

A key drawback of the classical approach presented in the previous section is that the mgf $M_{X_1}$ needs to exist, which naturally excludes heavy-tailed distributions.  In what follows, we focus on the classical Cram\'{e}r-Lundberg model, and we present another special case of Theorem~\ref{gform} that leads to a method for simulating $\psi(u)$ under the additional, yet mild, assumption $\E_P[X_1^2]{\,<\,}\infty$. From an actuarial perspective, the assumption $\E_P[X_1^2]{\,<\,}\infty$ is natural, since risks with infinite variance are not considered practically admissible.\smallskip

Fix  arbitrary $t,u{\,\in\,}\R_+$. Let $P_{W_1}{\,=\,}{\bf Exp}(\beta)$ for some real number $\beta{\,>\,}0$, and assume that  $\E_P[X_1^2]{\,<\,}\infty$. For any real number $\xi{\,<\,}0$, consider the \(\B(0,\infty)\)-measurable functions 
\[
\ga_\xi(x):=\ln\frac{1-\xi\cdot x}{1-\xi\cdot\E_P[X_1]}\,\,\text { for every } x\in (0,\infty),
\]
and 
\[
\delta_\xi(w):=\ln\big(1-\xi\cdot\E_P[X_1]\big)+\xi\cdot\beta\cdot\E_P[X_1]\cdot w\,\text{ for all } w\in(0,\infty).
\] 
It then follows that 
\[
\E_P\big[e^{\ga_\xi(X_1)}\big]=\E_P\big[e^{\delta_\xi(W_1)}\big]=1,
\]
and
\[
\E_P\big[X_1\cdot e^{\ga_\xi(X_1)}\big]=\frac{\E_P[X_1]-\xi\cdot\E_P[X^2_1]}{1-\xi\cdot\E_P[X_1]}<\infty\,\text{ and }\, \E_P\big[W_1\cdot e^{\delta_\xi(W_1)}\big]=\frac{1}{\beta\cdot (1-\xi\cdot\E_P[X_1])}<\infty,
\]
implying $(\ga_\xi,\delta_\xi){\,\in\,}\hyperlink{fp1}{\F^1_P}{\,\times\,}\hyperlink{gp1}{\G^1_P}$. Furthermore, it is easy to verify that 
\[
c\cdot \E_P\big[W_1\cdot e^{\delta_\xi(W_1)}\big]\leq\E_P\big[X_1\cdot e^{\ga_\xi(X_1)}\big] \Leftrightarrow	\xi\leq \frac{\beta\cdot\E_P[X_1]-c}{\beta\cdot\E_P[X_1^2]}=:\wh{\xi}<0,
\]
which implies that \((\ga_\xi,\delta_\xi){\,\in\,}\hyperlink{rp}{\C_P}\) for every \(\xi{\,\in\,}(-\infty, \wh{\xi}]\).	Thus,  we can apply Theorem~\ref{gform}(ii) to obtain a unique probability measure $Q^{(\xi)}{\,:=\,}Q^{(\ga_\xi,\delta_\xi)}{\,\in\,}\hyperlink{ms}{\M_S}$ determined by condition~\eqref{mq}, i.e.
\begin{gather*}
	Q^{(\xi)}(A)= \big(1-\xi\cdot\E_P[X_1]\big)\cdot\E_P\Big[\1_A\cdot e^{\xi\cdot\beta\cdot\E_P[X_1]\cdot T_{N_t+1}}\cdot\prod_{j=1}^{N_{t}}(1-\xi\cdot X_j)\Big]\,\,\text{ for all } A\in\H_t,
	\label{mmmdf}
\end{gather*}
and satisfying condition \eqref{mart1}. In particular, the ruin probability admits the representation \eqref{ruin} with  
\begin{gather*}  
	\psi(u)=\E_{Q^{(\xi)}}\Big[e^{-\xi\cdot\beta\cdot\E_P[X_1]\cdot \tau_u}\cdot\prod_{j=1}^{N_{\tau_u}}(1-\xi\cdot X_j)^{-1}\Big].
\end{gather*}
By Remark~\ref{remp}(a), the probability distributions of the claim sizes and the claim interarrival times under the new measure are given by
\begin{gather}
	Q^{(\xi)}_{X_1}(B)=\frac{1}{1-\xi\cdot\E_P[X_1]}\cdot\int_B\, P_{X_1}(dx)-\frac{ \xi\cdot \E_P[X_1]}{1-\xi\cdot\E_P[X_1]}\cdot \int_B \frac{x}{\E_P[X_1]}\,P_{X_1}(dx) 
	\label{mmmcsd}
\end{gather}
for every $B{\,\in\,}\B(0,\infty)$, and $Q^{(\xi)}_{W_1}{\,=\,}\textbf{Exp}(\beta_\xi)$, where $\beta_\xi{\,=\,}\beta\cdot(1-\xi\cdot\E_P[X_1]){\,\in\,}(0,\infty)$, respectively. Condition~\eqref{mmmcsd} shows that the claim size distribution under the new measure is a mixture of the original distribution and its size-biased counterpart (for more details on size-biased distributions and their applications in actuarial science see, for example, Bae \& Miljkovic~\cite{bami2024} and Denuit~\cite{de2019, de2020}).  \smallskip

The simulated results presented below are intended solely to illustrate the applicability of our method, rather than to assess the optimal choice of its parameters. Accordingly, the change of measure parameters were selected heuristically. To evaluate the numerical results, we use  the  \emph{absolute relative error} (ARE) and the \emph{relative standard error} (RSE) metrics, respectively given by
 \[
 \text{ARE}:=\left|\frac{\psi(u)-\Psi_K(u)}{\psi(u)} \right|,
 \]
 and
 \[
 \text{RSE}:=\text{RSE}\big(Q^{(\xi)}\big)=\frac{\sqrt{\mathbb{V}ar_{Q^{(\xi)}}[\Psi_K(u)]}}{\E_{Q^{(\xi)}}[\Psi_K(u)]}.
 \]
Additional diagnostics, including the \textit{effective sample size} and the \textit{maximum normalized weight}, were also monitored and led to conclusions consistent with the reported RSEs, but are not displayed for brevity.\smallskip 
 
 \textit{Throughout what follows,  the premium rate $c$ is of the form $c{\,=\,}(1+\eta)\cdot\frac{\E_P[X_1]}{\E_P[W_1]}$, where $\eta{\,\in\,}(0,\infty)$ is the safety loading.} \smallskip 

In our first example, we consider the case where the claim sizes are distributed according to the \textit{generalized gamma} law (abbreviated as \textbf{GGa}). Recall that a real-valued random variable $Y$ on $(\vO,\vS)$  follows the \textbf{GGa} distribution with parameters \((\al,b,p){\,\in\,}\R\setminus\{0\}{\,\times\,}(0,\infty)^2\) if
\[
P_Y(B)=\int_B\frac{|\al|\cdot x^{\al\cdot p-1}}{b^{\al\cdot p}\cdot\Gamma(p)}\cdot e^{-(x/b)^\al}\,\lambda(dx)\,\,\text{ for all } B\in\B(0,\infty),
\]
where $\Gamma{\,:\,}(0,\infty){\,\to\,}(0,\infty)$ is the \textit{gamma function}. Recall also that
\[
\E[Y^k]=\frac{b^k\cdot \Gamma(p+k/\al)}{\Gamma(p)}
\]
for all real numbers $k$ satisfying $p+\tfrac{k}{\al}{\,>\,}0$ (compare, e.g., Kleiber \& Kotz~\cite{klko03}, p.~151). The \textbf{GGa} distribution encompasses as special cases many well-known distributions used in actuarial practice including, but not limited to, the \textit{gamma} distribution ($\textbf{Ga}(p,b){\,=\,}\textbf{GGa}(1,1/b,p)$), the \textit{inverse gamma} distribution ($\textbf{InvGa}(p,b){\,=\,}\textbf{GGa}(-1,b,p)$), the \textit{Weibull} distribution ($\textbf{Wei}(\al,b){\,=\,}\textbf{GGa}(\al,b,1)$ for $\al{\,>\,}0$), and the \textit{inverse Weibull} distribution ($\textbf{InvWei}(\al,b){\,=\,}\textbf{GGa}(-\al,b,1)$ for $\al{\,>\,}0$). 
	
\begin{sex}\label{mmm1}\normalfont	
	
 Assume that $P_{X_1}{\,=\,}\textbf{GGa}(\al,b,p)$, where \((\al,b,p){\,\in\,}\R\setminus\{0\}\times (0,\infty)^2\) so that $p+2/\al{\,>\,}0$, and take $\xi{\,\in\,}(-\infty,\wh\xi]$, where 
		\[
		\wh\xi = \frac{\beta\cdot b\cdot \Gamma(\wt{p})-c\cdot \Gamma(p)}{\beta\cdot b^2\cdot \Gamma(p+2/\al)}\in(-\infty,0) 
		\]
		with $\wt{p}:=p+1/\al$.
		According to condition~\eqref{mmmcsd}, the claim size distribution (under $Q^{(\xi)}$) is given by the mixture
		\begin{align*}
			Q^{(\xi)}_{X_1}(B)&=\frac{\Gamma(p)}{\Gamma(p)-\xi\cdot b\cdot \Gamma(\wt{p})}\cdot \underbrace{\int_B \frac{|\al|\cdot x^{\al\cdot p-1}}{b^{\al\cdot p}\cdot\Gamma(p)}\cdot e^{-(x/b)^\al}\,\lambda(dx)}_{\textbf{GGa}(\al,b,p)}-\frac{\xi\cdot b\cdot \Gamma(\wt{p})}{\Gamma(p)-\xi\cdot b\cdot \Gamma(\wt{p})}\cdot \underbrace{\int_B \frac{|\al|\cdot x^{\al\cdot \wt{p}-1}}{b^{\al\cdot \wt{p}}\cdot\Gamma(\wt{p})}\cdot e^{-(x/b)^\al}\,\lambda(dx)}_{\textbf{GGa}(\al,b,\wt{p})}
		\end{align*}
		for every $B{\,\in\,}\B(0,\infty)$,  while for the claim interarrival times distribution we have  
		\[
		Q^{(\xi)}_{W_1} = \textbf{Exp}(\beta_\xi)\text{, where }
		\be_\xi:=\be\cdot\big(1- \xi\cdot b\cdot\tfrac{\Gamma(\wt{p})}{\Gamma(p)}\big)\in(0,\infty).
		\] 
		
Table~\ref{tableexp} presents simulated values of  $\Psi_{K}(u)$ (\(K{\,=\,}10^5\)) and the corresponding evaluation metrics (ARE and RSE) in the case of exponentially distributed claim sizes. To this purpose, take $\al{\,=\,}p{\,=\,}1$ and $b{\,=\,}1/\theta$, where $\theta{\,\in\,}(0,\infty)$. This implies that $P_{X_1}{\,=\,}\textbf{GGa}(1,1/\theta,1){\,=\,}\textbf{Exp}(\theta)$. For every real number 
		\[
		\xi \leq \wh\xi = \frac{\theta\cdot\beta-c\cdot\theta^2}{2\cdot\beta}<0,
		\]
		we get  that the distribution of the claim sizes under the new measure is given by the mixture
		\[ 
		Q^{(\xi)}_{X_1}(B)=\frac{\theta}{\theta-\xi}\cdot\underbrace{\int_B  \theta\cdot e^{-\theta\cdot x}\,\lambda(dx)}_{\textbf{Exp}(\theta)}-\frac{\xi}{\theta-\xi}\cdot \underbrace{\int_B \theta^2\cdot x\cdot e^{- \theta\cdot x}\,\lambda(dx)}_{\textbf{Ga}(2,\theta)}\,\text{ for all } B\in\B(0,\infty),
		\]
		while the distribution of the interarrival times becomes $Q^{(\xi)}_{W_1}{\,=\,}\textbf{Exp}(\beta_\xi)$, where $\beta_\xi{\,=\,}\beta\cdot(1-\xi/\theta){\,\in\,}(0,\infty)$. As in this case, the ruin probability satisfies
		\[
		\psi(u)=\frac{\beta}{\theta\cdot c}\cdot e^{-(\theta-\frac{\beta}{c})\cdot u}\,\,\text{ for every } u\in\R_+
		\]
		(compare  Rolski et al.~\cite{rss}, Corollary 6.5.3), this example provides a natural numerical benchmark. 
		\begin{table}[H]
			\captionsetup{format=plain, font=footnotesize, labelfont=bf, justification=raggedright, margin=0cm}\captionsetup{format=plain, font=footnotesize, labelfont=bf, justification=raggedright, margin=0cm}
			\caption{Comparison of the exact values for \(\psi(u)\) and the simulated values of \(\Psi_K(u)\) for exponentially distributed claim sizes with $\beta{\,=\,}\theta{\,=\,}1$, $\eta{\,=\,}1/2$, and \(\xi{\,=\,} 1.95\cdot\wh{\xi}\).
			} 
			\label{tableexp}
			\centering
			\vspace{0pt}
			\centering
			\begin{tabular}{c|c|c|c|c}
				\hline 
				$u$ & $\psi(u)$ & \(\Psi_K(u)\) & \(\text{ARE}\%\) & \(\text{RSE}\%\) \\
				\Xhline{1pt}
0 & \num{6.667e-01} & \num{6.667e-01} & 0.003\% & 0.091\% \\ 
1 & \num{4.777e-01} & \num{4.777e-01} & 0.007\% & 0.102\% \\ 
2 & \num{3.423e-01} & \num{3.419e-01} & 0.122\% & 0.119\% \\ 
3 & \num{2.453e-01} & \num{2.456e-01} & 0.159\% & 0.136\% \\ 
4 & \num{1.757e-01} & \num{1.759e-01} & 0.119\% & 0.155\% \\ 
5 & \num{1.259e-01} & \num{1.260e-01} & 0.098\% & 0.173\% \\ 
10 & \num{2.378e-02} & \num{2.393e-02} & 0.606\% & 0.254\% \\ 
20 & \num{8.484e-04} & \num{8.508e-04} & 0.280\% & 0.406\% \\ 
30 & \num{3.027e-05} & \num{3.036e-05} & 0.317\% & 0.556\% 
			\end{tabular}
		\end{table}
		
		Table~\ref{tableexp} indicates that the   estimates align with the exact values of \(\psi(u)\) across the considered reserve levels, with AREs below \(0.7\%\) even for probabilities of order \(10^{-5}\).\,The RSEs remain  small and increase only moderately with \(u\). These results establish an empirical baseline before applying the scheme to other members of the \(\mathbf{GGa}\) family, where exact formulas for \(\psi(u)\) are not available.\smallskip

The simulated values of $\Psi_{K}(u)$ (\(K{\,=\,}10^5\)) and the corresponding RSEs for several special cases of the \textbf{GGa} family are presented in Table~\ref{tablegga}, including both light- and heavy-tailed distributions. Note that these distributions are parametrized to have similar expectations $\E_P[X_1]{\,=\,}2$.
\begin{table}[H]
\captionsetup{format=plain, font=footnotesize, labelfont=bf, justification=raggedright, margin=0cm}
\caption{The simulated values of $\Psi_{K}(u)$ for \(\be{\,=\,}1\), \(\eta{\,=\,}1/2\)  and \(\xi{\,=\,} 1.95\cdot\wh{\xi}\) for some well-known members of the {\bf GGa} family with \(\E_P[X_1]{\,=\,}2\).}
\label{tablegga}
\centering
\sisetup{
scientific-notation = true,
table-format = 1.3e-2
}
\begin{tabular}{c|c|c|c|c}
\hline 
$u$ & {\textbf{Ga}(2,1)} & {\textbf{Wei}(3/4,1.68)} & {\textbf{InvGa}(3,4)} & {\textbf{InvWei}(3,1.48)} \\
\Xhline{1pt}
0 & \num{6.668e-01} {\footnotesize (0.083\%)} & \num{6.667e-01} {\footnotesize (0.101\%)}& 
\num{6.660e-01} {\footnotesize (0.170\%)}
& \num{6.654e-01} {\footnotesize (0.107\%)} \\ 
1 & \num{5.488e-01} {\footnotesize (0.086\%)} & \num{5.786e-01} {\footnotesize (0.106\%)} & 
\num{5.414e-01} {\footnotesize (0.206\%)}
& \num{5.342e-01} {\footnotesize (0.131\%)} \\ 
2 & \num{4.401e-01} {\footnotesize (0.095\%)} & \num{5.088e-01} {\footnotesize (0.113\%)} & 
\num{4.348e-01} {\footnotesize (0.253\%)}
& \num{4.022e-01} {\footnotesize (0.156\%)} \\ 
3 & \num{3.500e-01} {\footnotesize (0.106\%)} & \num{4.505e-01} {\footnotesize (0.120\%)} & 
\num{3.553e-01} {\footnotesize (0.304\%)}
& \num{3.101e-01} {\footnotesize (0.202\%)} \\ 
4 & \num{2.777e-01} {\footnotesize (0.117\%)} & \num{3.999e-01} {\footnotesize (0.128\%)} & 
\num{2.961e-01} {\footnotesize (0.640\%)}
& \num{2.431e-01} {\footnotesize (0.532\%)} \\ 
5 & \num{2.201e-01} {\footnotesize (0.127\%)} & \num{3.559e-01} {\footnotesize (0.135\%)} & 
\num{2.476e-01} {\footnotesize (0.517\%)}
& \num{1.906e-01} {\footnotesize (0.346\%)} \\ 
10 & \num{6.898e-02} {\footnotesize (0.171\%)} & \num{2.028e-01} {\footnotesize (0.175\%)} & 
\num{1.117e-01} {\footnotesize (0.576\%)}
& \num{6.489e-02} {\footnotesize (1.095\%)} \\ 
15 & \num{2.157e-02} {\footnotesize (0.211\%)} & \num{1.172e-01} {\footnotesize (0.216\%)} & 
\num{5.752e-02} {\footnotesize (1.787\%)}
& \num{2.602e-02} {\footnotesize (3.855\%)} \\ 
20 & \num{6.740e-03} {\footnotesize (0.246\%)} & \num{6.803e-02} {\footnotesize (0.255\%)} 
& \num{3.136e-02} {\footnotesize (1.428\%)}
& \num{1.134e-02} {\footnotesize (2.623\%)} \\
30 & \num{6.598e-04} {\footnotesize (0.315\%)}& \num{2.303e-02} {\footnotesize (0.353\%)} & 
\num{1.145e-02} {\footnotesize (2.627\%)}
& \num{3.229e-03} {\footnotesize (4.917\%)} \\ 
40 & \num{6.465e-05} {\footnotesize (0.384\%)}& \num{7.757e-03} {\footnotesize (0.441\%)} & 
\num{5.486e-03} {\footnotesize (5.038\%)}
& \num{1.175e-03} {\footnotesize (4.404\%)} \\ 
50 & \num{6.338e-06} {\footnotesize (0.456\%)} & \num{2.633e-03} {\footnotesize (0.550\%)} & 
\num{3.898e-03} {\footnotesize (22.290\%)}
&\num{5.778e-04} {\footnotesize (6.434\%)} 
\end{tabular}
\end{table}
 
In Table~\ref{tablegga}, we first observe an alignment between the simulated values at \(u{\,=\,}0\) and the theoretical value \(\psi(0){\,=\,}2/3\). For the \textbf{Ga} and \textbf{Wei} claim size distributions, \(\Psi_K(u)\) exhibits low RSEs over a wide range of initial reserves, indicating stable variance behaviour. In contrast, for the \textbf{InvGa} and \textbf{InvWei} distributions, the RSEs become unstable as \(u\) increases. This deterioration appears to be mainly due to the finite number of moments admitted by these distributions, rather than from their heavy-tailed nature alone (see also Example~\ref{PH1} below). 
\smallskip

Importantly, this deterioration in precision does not appear to be a limitation of the proposed method. For instance, in the \textbf{InvGa} case, by setting \(\xi {\,=\,} 2.14\cdot\wh{\xi}\), the RSE for \(u {\,=\,} 50\)  is reduced from \(22.29\%\) to \(4.695\%\). This suggests that the behaviour of the estimator may depend on \(u\). Consequently, a single fixed parameter choice may not yield uniformly efficient estimates across all reserve levels, particularly for claim size distributions with only finitely many moments. However, a systematic study of  adaptive parameter selection schemes is beyond the scope of this paper. 
\end{sex}  

Next, we consider the case of log-normally distributed claim sizes.  To this purpose, recall that a real-valued random variable $Y$ on $(\vO,\vS)$ follows the \textit{log-normal} distribution with parameters $\mu{\,\in\,}\R$ and $\sigma{\,\in\,}(0,\infty)$ ($\textbf{LN}(\mu,\sigma)$ for short) if
\[
P_{Y}(B)=\int_B \frac{1}{x}\cdot\frac{1}{\sigma\cdot\sqrt{2\cdot\pi}}\cdot e^{-\frac{(\ln x-\mu)^2}{2\cdot\sigma^2}}\,\lambda(dx)\,\text{ for every } B\in\B(0,\infty).
\]
The moments of the log-normal distribution are given by
\(
\E_P[Y^k]=e^{\mu\cdot k+\frac{\sigma^2\cdot k^2}{2}}
\)
for all real numbers $k{\,>\,}0$.  Note that in this case the standard simulation of $\psi(u)$ via the PK formula is challenging, since the distribution function of $\textbf{LN}(\mu,\sigma)$ lacks a closed analytical expression (see Juneja \& Shahabuddin~\cite{jusha02}, p.~112).  

\begin{sex}\label{mmm2}\normalfont	
 Assume that $P_{X_1}{\,=\,}\textbf{LN}(\mu,\sigma)$ for some  $\mu{\,\in\,}\R$ and $\sigma{\,\in\,}(0,\infty)$, and let $\xi{\,\in\,}(-\infty,\wh{\xi}]$, where  
		\[
		\wh\xi=\frac{\beta-c\cdot e^{-\mu-\frac{\sigma^2}{2}}}{\beta}\cdot e^{-\mu-\frac{3\cdot\sigma^2}{2}}\in(-\infty,0). 
		\]
		For this particular case, we get  by condition~\eqref{mmmcsd} that the claim sizes are distributed according to the  mixture 
		\begin{align*}
			Q^{(\xi)}_{X_1}(B)&=\frac{1}{1-\xi\cdot e^{\mu+\frac{\sigma^2}{2}}}\cdot\underbrace{\int_B \frac{1}{x}\cdot\frac{1}{\sigma\cdot\sqrt{2\cdot\pi}}\cdot e^{-\frac{(\ln x-\mu)^2}{2\cdot\sigma^2}}\,\lambda(dx)}_{\textbf{LN}(\mu,\sigma)} - \frac{\xi\cdot e^{\mu+\frac{\sigma^2}{2}}}{1-\xi\cdot e^{\mu+\frac{\sigma^2}{2}}}\cdot\underbrace{\int_B\frac{1}{x}\cdot\frac{1}{\sigma\cdot\sqrt{2\cdot\pi}}\cdot e^{-\frac{(\ln x-\wt{\mu})^2}{2\cdot\sigma^2}}\,\lambda(dx)}_{\textbf{LN}(\wt{\mu},\sigma)} 
		\end{align*}
		for every $B{\,\in\,}\B(0,\infty)$, where $\wt{\mu}{\,:=\,}\mu+\sigma^2$. Furthermore, for the claim interarrival times we get that 
		\[Q^{(\xi)}_{W_1}=\textbf{Exp}(\beta_\xi)\text{, where } 
		\beta_\xi=\beta\cdot\big(1-\xi\cdot e^{\mu+\frac{\sigma^2}{2}}\big)\in(0,\infty).
		\]
		 
 In Table~\ref{tableln}, we present simulated values of  $\Psi_{K}(u)$ (\(K{\,=\,}10^5\))   and the corresponding RSEs  for the case of log-normally distributed claim sizes.  
\begin{table}[H] 
\captionsetup{format=plain, font=footnotesize, labelfont=bf, justification=raggedright, margin=0cm}
\caption{The simulated values of $\Psi_{K}(u)$ for log-normally distributed claim sizes with \(\be{\,=\,}1\), \(\eta{\,=\,}1/2\), \(\mu{\,=\,}0\), \(\sigma{\,\in\,}\{1/2,1,3/2\}\), and \(\xi{\,=\,} 1.95\cdot\wh{\xi}\).}
\label{tableln}
\centering
\sisetup{
scientific-notation = true,
table-format = 1.3e-2
}
\begin{tabular}{c | c|c|c}
\hline 
$u$ & {\(\textbf{LN}(0,1/2)\)} & {\(\textbf{LN}(0,1)\)} & {\(\textbf{LN}(0,3/2)\)}\\
\Xhline{1pt}
0  &  \num{6.664e-1} {\footnotesize (0.080\%)} & \num{6.661e-1} {\footnotesize (0.123\%)} & \num{6.663e-1}  {\footnotesize (0.218\%)} \\ 
1  & \num{4.377e-1}  {\footnotesize (0.094\%)} & \num{5.419e-1}  {\footnotesize (0.142\%)} & \num{6.106e-1}  {\footnotesize (0.237\%)} \\  
2  & \num{2.690e-1} {\footnotesize (0.116\%)} & \num{4.518e-1}  {\footnotesize (0.168\%)} & \num{5.724e-1 } {\footnotesize (0.256\%)} \\ 
3  & \num{1.656e-1}  {\footnotesize (0.138\%)} & \num{3.831e-1}  {\footnotesize (0.193\%)} & \num{5.414e-1}  {\footnotesize (0.271\%)} \\ 
4  & \num{1.021e-1}  {\footnotesize (0.159\%)} & \num{3.281e-1}  {\footnotesize (0.217\%)} & \num{5.163e-1}  {\footnotesize (0.284\%)} \\ 
5  & \num{6.292e-2} {\footnotesize (0.178\%)} & \num{2.834e-1} {\footnotesize (0.225\%)} & \num{4.943e-1}  {\footnotesize (0.301\%)} \\ 
10 & \num{5.646e-3}  {\footnotesize (0.271\%)} & \num{1.443e-1}  {\footnotesize (0.375\%)} & \num{4.091e-1}  {\footnotesize (0.384\%)} \\ 
15 & \num{5.076e-4}  {\footnotesize (0.378\%)} & \num{7.881e-2}  {\footnotesize (0.579\%)} & \num{3.505e-1}  {\footnotesize (0.437\%)} \\ 
20 & \num{4.541e-5}  {\footnotesize (0.449\%)} & \num{4.512e-2}  {\footnotesize (0.958\%)} & \num{3.051e-1}  {\footnotesize (0.491\%)} \\ 
30 & \num{3.694e-7}  {\footnotesize (0.690\%)} & \num{1.606e-2} {\footnotesize (1.694\%)} & \num{2.387e-1}  {\footnotesize (0.433\%)} \\ 
40 & \num{2.933e-9}  {\footnotesize (0.861\%)} & \num{6.164e-3}  {\footnotesize (1.883\%)} & \num{1.926e-1}  {\footnotesize (0.466\%)} \\ 
50 & \num{2.390e-11} {\footnotesize (1.245\%)} & \num{2.509e-3}  {\footnotesize (2.277\%)} & \num{1.593e-1}  {\footnotesize (0.570\%)} 
\end{tabular}
\end{table}

In Table~\ref{tableln}, we again observe a close agreement between the simulated values at \(u{\,=\,} 0\) and the exact value \(\psi(0){\,=\,} 2/3\). Furthermore, \(\Psi_K(u)\) exhibits consistently low RSEs across all values of \(u\) considered, indicating stable behaviour of the estimator.\,This effect is most clearly observed for \({\bf LN}(0,1/2)\), where the RSEs remain uniformly small and increase only very mildly with \(u\), even for probabilities of order \(10^{-11}\).\,While a more noticeable increase in RSEs is observed for \({\bf LN}(0,1)\), this increase remains moderate, and the overall behaviour is still stable.
\end{sex}

	\begin{srem}\label{remmmm}\normalfont 
		If we choose $\xi{\,=\,}\wh\xi$, then the restriction $\wh{Q}_T$ of the probability measure $\wh{Q}{\,:=\,}Q^{(\wh\xi)}$ to $\F_T$, where $T{\,<\,}\infty$ is a fixed time horizon,  coincides with the \textit{minimal  martingale measure} (MMM) for CPPs (see M\o{}ller~\cite{moller04}, Section~4.2) for the price process $\{u-R^u_t\}_{t\in[0,T]}$. Under the MMM, the  no-arbitrage  premium is  
		\[
		\wh{c}:=c(\wh{\xi})=\beta\cdot\big(\E_P[X_1]-\wh{\xi}\cdot\E_P[X_1^2]\big).
		\] 
	\end{srem}

\subsection{The hazard rate twisting approach}\label{PH}

Although the approach presented in the previous section provides a fairly general method for simulating ruin probabilities in the Cram\'er-Lundberg risk model, the assumption $\E_P[X^2_1]{\,<\,}\infty$ excludes cases of particular interest in insurance mathematics (e.g., a Pareto distribution with shape parameter in $(1,2]$). In addition,  this method cannot be extended in a straightforward manner to the Sparre Andersen model, as in this case the measure $\widehat{Q}_T$ fails, in general, to be an EMM for the price process $\{u-R^u_t\}_{t\in[0,T]}$. These limitations motivate us to consider the well-known \textit{proportional hazard} (PH) transform applied to $P_{X_1}$ and/or $P_{W_1}$ (see Wang~\cite{wa}), in order to increase the weight assigned to unfavourable events. In the classical Cram\'{e}r-Lundberg setting,  the PH-transform approach has been considered in Juneja \& Shahabuddin~\cite{jusha02} under the name \textit{hazard rate twisting}, where it is used to modify (only) the claim size distribution and construct IS  estimators for $\psi(u)$ based on the PK formula.\smallskip 

 Recall that the \textit{hazard rate function} $h_Y$ of a positive real-valued random variable $Y$ on $\vO$ with $P_Y{\,\ll\,}\lambda$  is defined as  
\[
h_Y(x):=\frac{f_Y(x)}{1-F_Y(x)}\,\text{ for every } x\in (0,\infty),
\]
where $F_Y$ denotes the probability distribution function of \(Y\) and $f_Y{\,:=\,}dP_Y/d\lambda$ is the corresponding probability density function (cf., e.g., Rolski et al.~\cite{rss}, Section 2.4.2). For every \(x{\,\in\,}(0,\infty)\) define 
\[
H_Y(x) := \int_0^x h_Y(t)\,\lambda(dt).
\]

Let $t,u{\,\in\,}\R_+$ be fixed but arbitrary. For every $(r,\theta){\,\in\,}(0,\infty)^2$ consider the \(\B(0,\infty)\)-measurable functions 
\[
\ga_r(x):= \ln r-(r-1)\cdot H_{X_1}(x),\,\, \text{ for every } x\in(0,\infty), 
\]
and 
\[
\delta_\theta(w):= \ln \theta-(\theta-1)\cdot H_{W_1}(w)\,\, \text{ for all } w\in (0,\infty) .
\]
Clearly, \((\ga_r,\delta_\theta){\,\in\,}\hyperlink{fp}{\F_P}{\,\times\,}\hyperlink{gp}{\G_P}\) for all \((r,\theta){\,\in\,}(0,\infty)^2\). Assume  that there exists a set $A_{X_1,W_1}{\,\subseteq\,}(0,\infty)^2$ such that  \((\ga_r,\delta_\theta){\,\in\,}\hyperlink{rp}{\C_P}\) for every \((r,\theta){\,\in\,} A_{X_1,W_1}\).  Thus, for any chosen \((r,\theta){\,\in\,} A_{X_1,W_1}\), we can apply Theorem~\ref{gform}(ii) to obtain a unique probability measure $Q^{(r,\theta)}{\,:=\,}Q^{(\ga_r,\delta_\theta)}{\,\in\,}\hyperlink{rms}{\mathcal{R}_S}$ determined by condition~\eqref{mq}, namely
\[
Q^{(r,\theta)}(A)=\E_P\Big[\1_A\cdot r^{N_t}\cdot\theta^{N_{t}+1}\cdot e^{-(r-1)\cdot\sum_{j=1}^{N_t} H_{X_1}(X_j)-(\theta-1)\cdot\sum_{j=1}^{N_t+1} H_{W_1}(W_j)}\Big] \,\,\text{ for all } A\in\H_t,
\]
and satisfying condition \eqref{mart1}. In particular, the ruin probability \(\psi(u)\) admits the representation \eqref{ruin} with  
\begin{gather}
	\psi(u)=\E_{Q^{(r,\theta)}}\Big[(r\cdot\theta)^{-N_{\tau_u}}\cdot e^{(r-1)\cdot\sum_{j=1}^{N_{\tau_u}} H_{X_1}(X_j)+(\theta-1)\cdot\sum_{j=1}^{N_{\tau_u}} H_{W_1}(W_j)}\Big].
	\label{ruinph}
\end{gather}
By Remark~\ref{remp}(a) we have that 
\begin{gather}
	Q^{(r,\theta)}_{X_1}(B)=\int_B r\cdot h_{X_1}(x)\cdot e^{-r\cdot H_{X_1}(x)}\,\lambda(dx)\,\text{ for every } B\in\B(0,\infty),
	\label{csdph}
\end{gather}
and
\begin{gather}
	Q^{(r,\theta)}_{W_1}(F)=\int_F \theta\cdot h_{W_1}(w)\cdot e^{-\theta\cdot H_{W_1}(w)}\,\lambda(dw)\,\text{ for all } F\in\B(0,\infty),
	\label{citdph}
\end{gather}
respectively.\smallskip

In the next example, we consider the case of Pareto distributed claim sizes.  To this purpose, recall that a real-valued random variable $Y$ on $(\vO,\vS)$ follows the \textit{Pareto} distribution with parameters $a,b{\,\in\,}(0,\infty)$ ($\textbf{Pa}(a,b)$ for short) if
\[
P_{Y}(B)=\int_B \frac{a\cdot b^a}{(b+x)^{a-1}}\,\lambda(dx)\,\text{ for every } B\in\B(0,\infty).
\]

\begin{sex}\label{PH1}
\normalfont

Let  $P_{W_1}{\,=\,}\textbf{Wei}(\alpha,\beta)$ and \(P_{X_1}{\,=\,}\textbf{Pa}(a,b)\), where $\al,\be, b{\,\in\,}(0,\infty)$ and  \(a{\,\in\,}(1,\infty)\). Then
\[
h_{W_1}(w)=\alpha\cdot\be^{-\al}\cdot w^{\al-1} \,\,\text{ and }\,\, H_{W_1}(w)=\be^{-\alpha}\cdot w^\alpha
\]
for all $w{\,\in\,}(0,\infty)$, and 
\[
h_{X_1}(x)=\frac{a}{b+x}\,\,\text{ and }\,\, H_{X_1}(x)=a\cdot\ln\Big(1+\frac{x}{b}\Big) 
\]
for every $x{\,\in\,}(0,\infty)$.  It then follows that 
\begin{align*}
\E_P\big[W_1\cdot e^{\delta_\theta(W_1)}\big] 
&=\int_0^\infty w\cdot\alpha\cdot\big(\be\cdot\theta^{-1/\alpha}\big)^{-\al}\cdot w^{\al-1}\cdot e^{-\big(\be\cdot\theta^{-1/\alpha}\big)^{-\al}\cdot w^\alpha}\,\lambda(dw) =\theta^{-1/\alpha}\cdot\be\cdot \Gamma(1+1/\alpha)<\infty
\end{align*}
for any $\theta{\,\in\,}(0,\infty)$, and 
\[
\E_P\big[X_1\cdot e^{\ga_r(X_1)}\big]=\int_0^\infty x\cdot \frac{a\cdot r}{b+x}\cdot e^{-a\cdot r\cdot\ln(1+\frac{x}{b})}\,\lambda(dx)=\int_0^\infty x\cdot \frac{a\cdot r\cdot b}{(b+x)^{a\cdot r+1}}\,\lambda(dx)=\frac{b}{a\cdot r-1}<\infty
\]
for every \(r{\,\in\,}(1/a,\infty)\); hence  $(\ga_r,\delta_\theta){\,\in\,}\hyperlink{fp1}{\F^1_P}{\,\times\,}\hyperlink{gp1}{\G^1_P}$ for all $(r,\theta){\,\in\,}(1/a,\infty){\,\times\,}(0,\infty)$. Solving the inequality 
\[
c\cdot \E_P\big[W_1\cdot e^{\delta_\theta(W_1)}\big]\leq\E_P\big[X_1\cdot e^{\ga_r(X_1)}\big]
\]
yields that \((\ga_r,\delta_\theta){\,\in\,}\hyperlink{rp}{\C_P}\) if and only if \((r,\theta){\,\in\,} A_{X_1,W_1}\), where 
\[
A_{X_1,W_1}=\big\{(r,\theta)\in (1/a,\infty)\times(0,\infty) : r\leq r_M(\theta)\big\}=\big\{(r,\theta)\in (1/a,\infty)\times(0,\infty) : \theta\geq \theta_m(r)\big\}
\]
with
\[
r_M(\theta):=\frac{1}{a}\cdot \Big(1+\frac{b\cdot\theta^{1/\alpha}}{c\cdot\E_P[W_1]}\Big)\,\,\text{ for every } \theta\in(0,\infty),
\]
and 
\[
\theta_m(r):=\frac{c^\al}{b^\al}\cdot\big((a\cdot r-1)\cdot\E_P[W_1]\big)^\alpha\,\,\text{ for every } r\in(1/a,\infty). 
\]
Thus, for every $(r,\theta){\,\in\,}A_{X_1,W_1}$  we can define the probability measure
\[
Q^{(r,\theta)}(A)=\E_P\Big[\1_A\cdot r^{N_{t}}\cdot \theta^{N_{t}+1}\cdot e^{-(\theta-1)\cdot \beta^{-\al}\cdot\sum_{j=1}^{N_t+1}W_j^{\alpha}}\cdot\prod_{j=1}^{N_{t}}\Big(1+\frac{X_j}{b}\Big)^{-a\cdot(r-1)} \Big]\,\,\text{ for all } A\in\H_t.
\] 
Condition~\eqref{ruinph} can be rewritten as
\[
\psi(u)=\E_{Q^{(r,\theta)}}\Big[(r\cdot\theta)^{-N_{\tau_u}}\cdot e^{(\theta-1)\cdot \beta^{-\al}\cdot\sum_{j=1}^{N_{\tau_u}}W_j^{\alpha}}\cdot\prod_{j=1}^{N_{\tau_u}}\Big(1+\frac{X_j}{b}\Big)^{a\cdot(r-1)} \Big],
\]
while conditions~\eqref{csdph} and~\eqref{citdph} yield $Q^{(r,\theta)}_{X_1}{\,=\,}\textbf{Pa}(a\cdot r,b)$ and $Q^{(r,\theta)}_{W_1}{\,=\,}\textbf{Wei}(\al,\be\cdot\theta^{-1/\al})$, respectively.  In Table~\ref{tablepawe}, we present the simulated values of  $\Psi_{K}(u)$ (\(K{\,=\,}10^5\))  and the corresponding RSEs.
\begin{table}[H] 
	\captionsetup{format=plain, font=footnotesize, labelfont=bf, justification=raggedright, margin=0cm}
	\caption{The simulated values of $\Psi_{K}(u)$ for \(P_{X_1}{\,=\,}\textbf{Pa}(a,3)\), \(a{\,\in\,}\{3/2,2,5/2\}\),    \(P_{W_1}{\,=\,}\textbf{Wei}(0.375,1/2)\), $\eta{\,=\,}1/2$,   \(\theta{\,=\,}1.2\), and \(r{\,=\,}0.95\cdot r_M(\theta)\). }
	\label{tablepawe}
	\centering
	\sisetup{
		scientific-notation = true,
		table-format = 1.3e-2
	}  
	\begin{tabular}{c | c|c|c}
\hline  
$u$ & {\(\textbf{Pa}(3/2,3)\)} & {\(\textbf{Pa}(2,3)\)} & {\(\textbf{Pa}(5/2,3)\)}\\
\Xhline{1pt}
0  &  \num{8.881e-01} {\footnotesize (0.163\%)} & \num{9.125e-01} {\footnotesize (0.145\%)} & \num{9.177e-01} {\footnotesize (0.116\%)} \\ 
5  & \num{7.845e-01} {\footnotesize (0.314\%)} & \num{7.807e-01} {\footnotesize (0.371\%)} & \num{7.504e-01} {\footnotesize (0.360\%)} \\  
10  & \num{7.190e-01} {\footnotesize (0.367\%)}  & \num{6.947e-01} {\footnotesize (0.422\%)} & \num{6.380e-01} {\footnotesize (0.451\%)} \\ 
15  & \num{6.701e-01} {\footnotesize (0.395\%)}  & \num{6.257e-01} {\footnotesize (0.458\%)} & \num{5.482e-01} {\footnotesize (0.510\%)} \\ 
20  & \num{6.287e-01} {\footnotesize (0.422\%)} & \num{5.669e-01} {\footnotesize (0.498\%)} & \num{4.736e-01} {\footnotesize (0.530\%)} \\ 
30  & \num{5.649e-01} {\footnotesize (0.505\%)} & \num{4.766e-01} {\footnotesize (0.693\%)} & \num{3.600e-01} {\footnotesize (0.715\%)} \\ 
40 & \num{5.160e-01} {\footnotesize (0.684\%)} & \num{4.067e-01} {\footnotesize (0.994\%)} & \num{2.755e-01} {\footnotesize (0.783\%)} \\ 
50 & \num{4.738e-01} {\footnotesize (0.723\%)} & \num{3.463e-01} {\footnotesize (0.860\%)} & \num{2.133e-01} {\footnotesize (0.951\%)} \\ 
100 & \num{3.426e-01} {\footnotesize (1.213\%)}  & \num{1.776e-01} {\footnotesize (1.906\%)} & \num{6.643e-02} {\footnotesize (2.165\%)} \\ 
150 & \num{2.625e-01} {\footnotesize (1.488\%)} & \num{1.020e-01} {\footnotesize (3.052\%)} & \num{2.317e-02} {\footnotesize (3.718\%)} \\ 
200 & \num{2.167e-01} {\footnotesize (2.065\%)} & \num{6.308e-02} {\footnotesize (3.234\%)} & \num{9.804e-03} {\footnotesize (4.435\%)} \\ 
250 & \num{1.792e-01} {\footnotesize (1.933\%)} & \num{4.509e-02} {\footnotesize (6.408\%)} & \num{5.798e-03} {\footnotesize (14.482\%)}
\end{tabular}
\end{table}

In Table~\ref{tablepawe}, we observe that the parameter $a$ significantly affects the decay rate of the ruin probability. Higher values of $a$ (which correspond to lighter tails in the claim size distribution) lead to a faster reduction in ruin probability as $u$ increases. For $u{\,=\,}250$, the probability for $a{\,=\,}5/2$ is approximately $31$ times smaller than for $a{\,=\,}3/2$. The RSEs indicate that the simulation is accurate for small  values of $u$. However, as $u$ increases and $\Psi_{K}(u)$ becomes smaller, the relative error increases, reaching $14.482\%$ in the most extreme case ($a{\,=\,}5/2, u{\,=\,}250$), which is possibly due to the selected parameter values. In fact, if we choose \(\theta{\,=\,}1.1\) and \(r{\,=\,}0.875\cdot r_M(\theta)\), then the RSE reduces from $14.482\%$ to \(5.758\%\). This, together with Example~\ref{mmm1}, reinforces the intuition that, especially for heavy-tailed distributions with a finite number of moments, the selection of the parameters should depend on the level of the initial reserve.
\end{sex}

 In some cases, a closed-form expression for the hazard function is not available (e.g. the \textbf{Ga} or the \textbf{LN} distributions). However, the general method presented above is still applicable as long as at least one of \(h_{X_1}\) or \(h_{W_1}\) is available in closed form, since one may twist only one component while leaving the  other unchanged (taking \(r{\,=\,}1\) or \(\theta{\,=\,}1\)).

\begin{sex}\label{PH2}
	\normalfont		

 Assume that \(P_{W_1}{\,=\,}\textbf{Ga}(\al,\beta)\) and \(P_{X_1}{\,=\,}\textbf{Exp}(\zeta)\), where \(\al,\be,\zeta{\,\in\,}(0,\infty)\), and take \((r,\theta){\,\in\,}(0,\infty){\,\times\,}\{1\}\). This implies that \(\delta_\theta(w){\,=\,}\delta_1(w){\,=\,}0\) for all \(w{\,\in\,}(0,\infty)\), and thus we change only the claim size distribution. Since  \(P_{X_1}{\,=\,}\textbf{Exp}(\zeta)\), we have
		\[
		h_{X_1}(x)=\zeta \,\,\text{ and }\,\, H_{X_1}(x)=\zeta \cdot x
		\]
		for all \(x{\,\in\,}(0,\infty)\). We then get that 
		\[
		\E_P\big[W_1\cdot e^{\delta_1(W_1)}\big] = \E_P[W_1]=\frac{\al}{\be}<\infty
		\]
		and 
		\[
		\E_P\big[X_1\cdot e^{\ga_r(X_1)}\big]=\int_0^\infty x\cdot r\cdot\zeta\cdot e^{- r\cdot\zeta\cdot x}\,\lambda(dx)=\frac{1}{r\cdot\zeta}<\infty
		\]
		for every \(r{\,\in\,}(0,\infty)\); hence $(\ga_r,\delta_\theta){\,\in\,}\hyperlink{fp1}{\F^1_P}{\,\times\,}\hyperlink{gp1}{\G^1_P}$ for all $(r,\theta){\,\in\,}(0,\infty){\,\times\,}\{1\}$. Moreover, as \(\theta{\,=\,1}\), we get
		\[
		c\cdot \E_P\big[W_1\cdot e^{\delta_\theta(W_1)}\big]\leq\E_P\big[X_1\cdot e^{\ga_r(X_1)}\big] \Leftrightarrow	r\leq \frac{\beta}{c\cdot\al\cdot\zeta}<1,
		\]
		implying
		\[
		(\ga_r,\delta_\theta)\in\hyperlink{rp}{\C_P}\Leftrightarrow (r,\theta)\in A_{X_1,W_1}=\big\{(r,\theta)\in (0,\infty)\times\{1\} :   r\leq \tfrac{\beta}{c\cdot\al\cdot\zeta}\big\}.
		\]
		Thus, for every $(r,\theta){\,\in\,}A_{X_1,W_1}$  we can define the probability measure
		\[
		Q^{(r)}(A):= Q^{(r,1)}(A)=\E_P\Big[\1_A\cdot r^{N_{t}}\cdot   e^{-\zeta\cdot(r-1)\cdot S_t}\Big]\,\,\text{ for all } A\in\H_t.
		\] 
		Condition~\eqref{ruinph} can be rewritten as
		\[
		\psi(u)=\E_{Q^{(r)}}\Big[r^{-N_{\tau_u}}\cdot e^{\zeta\cdot(r-1)\cdot S_{\tau_u}} \Big],
		\]
		while conditions~\eqref{csdph} and~\eqref{citdph} yield $Q^{(r)}_{X_1}{\,=\,}\textbf{Exp}(r\cdot\zeta)$ and $Q^{(r)}_{W_1}{\,=\,} P_{W_1}$, respectively. Recall that in the Sparre Andersen model with exponentially distributed claim sizes, the ruin probability is given by the formula
		\[
		\psi(u)=\Big(1-\frac{\rho}{\zeta}\Big) \cdot e^{-\rho\cdot u} 
		\]
		(cf., e.g.  Rolski et al.~\cite{rss}, Corollary 6.5.2), where \(\rho\) is the unique positive solution to the equation \(\ka(r){\,=\,}0\) (see Section~\ref{SALT}). 
 	\begin{table}[H] 
			\captionsetup{format=plain, font=footnotesize, labelfont=bf, justification=raggedright, margin=0cm}
			\caption{Comparison of the exact values for \(\psi(u)\) and the simulated values of \(\Psi_K(u)\) for \(P_{X_1}{\,=\,}{\bf Exp}(1)\), \(P_{W_1}{\,=\,}{\bf Ga}(2,1)\), $\eta{\,=\,}1/2$, \(\theta{\,=\,}1\), and \(r{\,=\,}0.9\cdot r_M(\theta)\).}
			\label{tableexgaph}
			\centering
			\sisetup{
				scientific-notation = true,
				table-format = 1.3e-2
			}
			\begin{tabular}{c | c|c|c|c}
				\hline 
				$u$ & {$\psi(u)$} & {$\Psi_K(u)$} &  \(\text{ARE}\%\) & \(\text{RSE}\%\)\\
				\Xhline{1pt}
				0 & \num{5.750e-01} & \num{5.751e-01} & 0.008\% & 0.255\% \\ 
				1 & \num{3.759e-01} & \num{3.767e-01} & 0.193\% & 0.326\% \\ 
				2 & \num{2.458e-01} & \num{2.460e-01} & 0.094\% & 0.390\% \\ 
				3 & \num{1.607e-01} & \num{1.614e-01}& 0.465\% & 0.550\% \\ 
				4 & \num{1.051e-01} & \num{1.058e-01} & 0.663\% & 0.529\% \\ 
				5 & \num{6.869e-02} & \num{6.923e-02} & 0.784\% & 0.628\% \\ 
				10 & \num{8.205e-03} & \num{8.202e-03} & 0.039\% & 0.963\% \\ 
				20 & \num{1.171e-04} & \num{1.174e-04} & 0.283\% & 2.108\% \\ 
				30 & \num{1.670e-06} & \num{1.664e-06} & 0.397\% & 4.480\% 
			\end{tabular}
		\end{table}
	
	Table~\ref{tableexgaph} indicates that \(\Psi_K(u)\) estimates align with the exact values of \(\psi(u)\) across the considered reserve levels, with AREs below \(0.8\%\) even for probabilities of order \(10^{-6}\). The RSEs remain small for moderate values of \(u\) but increase more noticeably as \(u\) increases, reaching \(4.480\%\) at \(u{\,=\,}30\).
	
	\end{sex}

\section{Conclusion and Discussion} \label{cr}
	
In this paper, given an initial probability measure \(P\) under which \(S\) is a \(P\)-CRP and \(\E_P[X_1]{\,<\,}\infty\), we  established a complete characterization of the class \(\hyperlink{rms}{\mathcal{R}_S}\) of ruin-inducing probability measures \(Q\) that are progressively equivalent to \(P\) and preserve the structure of \(S\). This characterization is expressed in terms of tilting pairs \((\ga,\delta){\,\in\,}\hyperlink{rp}{\C_P}\) (see Theorem~\ref{gform}). As a consequence, we obtained a flexible and unified representation of the infinite-time ruin probability \(\psi(u)\) that extends the classical adjustment coefficient framework and applies to both light- and heavy-tailed claim size and interarrival time distributions (see \eqref{ruin} in Theorem~\ref{gform}).  This work is not intended as an exhaustive treatment of the subject. It rather provides a conceptual and methodological foundation for further developments in modern risk theory and rare event simulation. Although our results provide a practical framework for estimating ruin probabilities in the Sparre Andersen risk model via IS, several fundamental questions about the efficiency of the simulation remain open.\smallskip

It is well known that, for every initial capital \( u{\,\in\,}\R_+ \), the estimator \( e^{-\rho \cdot(u - R_{\tau_u})} \), appearing in Section~\ref{SALT}, exhibits bounded relative error (see Asmussen \& Albrecher~\cite{asal}, Theorem~3.3). Since this estimator corresponds to a specific choice of \( (\ga,\delta) {\,\in\,} \hyperlink{rp}{\C_P} \), it is natural to ask whether other members of \( \hyperlink{rp}{\C_P} \) also satisfy this property.\smallskip

For light-tailed claim size distributions, the classical change of measures approach based on the Esscher transform (see Section~\ref{SALT}) yields the celebrated Cram\'{e}r-Lundberg approximation (cf., e.g., Schmidli~\cite{scm}, Sections~8.2 and~8.3). It would be of considerable interest to determine whether the Embrechts-Veraverbeke approximation for \( \psi(u) \) in the subexponential case (see Embrechts \& Veraverbeke~\cite{ev82}) can likewise be recovered through a suitable choice of \( (\gamma,\delta) {\,\in\,} \hyperlink{rp}{\C_P} \). Such a pair would constitute an appropriate candidate for constructing a log-efficient IS scheme for estimating \( \psi(u) \) in the heavy-tailed regime.\smallskip

In the examples presented in Sections~\ref{mmm} and~\ref{PH}, the simulation parameters were selected heuristically, which raises the question of identifying their optimal values. More generally, let \( Q {\,\in\,} \hyperlink{rms}{\mathcal{R}_S} \). By Theorem~\ref{gform}(i), there exists an essentially unique pair \( (\ga,\delta) {\,\in\,} \hyperlink{rp}{\C_P} \) satisfying condition~\eqref{ruin}. An application of the Monotone Convergence Theorem yields
\begin{align*}
	\mathbb{V}ar_Q\Big[e^{-S_{\tau_u}^{(\ga)}-T^{(\delta)}_{N_{\tau_u}}}\Big] 
	&=\E_P\Big[\1_{\{\tau_u<\infty\}}\cdot e^{-S_{\tau_u}^{(\ga)}-T^{(\delta)}_{N_{\tau_u}}}\Big]-\big(\psi(u)\big)^2\,\,\text{ for every } u\in\R_+.
\end{align*}
Since \(\psi(u)\) is independent of the pair \( (\ga,\delta) \), the above equality  leads to the optimization problem of determining whether there exists a pair 
\( (\ga^\ast,\delta^\ast) {\,\in\,} \hyperlink{rp}{\C_P} \) minimizing
\[
\E_P\Big[\1_{\{\tau_u<\infty\}}\cdot e^{-S_{\tau_u}^{(\ga)}-T^{(\delta)}_{N_{\tau_u}}}\Big].
\]
 
Another promising direction for future research is the formal embedding of existing rare-event simulation algorithms into the  class \( \hyperlink{rp}{\C_P} \). For instance, IS methods based on the concepts of \textit{delayed hazard rate twisting} and \textit{weighted delayed hazard rate twisting} (see Juneja~\&~Shahabuddin~\cite{jusha02} for details) lead to estimators with improved performance compared to the one discussed in Section~\ref{PH}. Identifying the functional forms corresponding to these known methods would not only provide a rigorous martingale foundation for their effectiveness, but also facilitate the derivation of new IS schemes.\smallskip

Finally, an interesting topic for future research is the extension of the current methodology to more general risk models that relax classical assumptions, such as settings where interarrival times and claim sizes are dependent or not identically distributed (see, e.g., Albrecher \& Boxma~\cite{albo04}, Cossette et~al.~\cite{colama15}, Lyberopoulos \& Macheras~\cite{lm3} and Kizinevič \& Šiaulys~\cite{kisi18}). In a forthcoming paper, we will explore in greater depth the connection between fair pricing in insurance (premium calculation principles) within an arbitrage-free market and the ruin probability of the corresponding reserve process in a Sparre Andersen risk model, building on the insights of Delbaen \& Haezendonck~\cite{dh} and Macheras \& Tzaninis~\cite{mt3}.

\end{document}